\documentclass[12pt]{amsart}
\usepackage{amssymb,fullpage}
\usepackage{times}
\usepackage{amsmath,amsfonts,amsthm}
\usepackage{mathtools,sansmath} 
\usepackage{enumerate,wrapfig}
\usepackage{mathrsfs,multirow,makecell}
\usepackage[table]{xcolor}
\usepackage{comment}
\usepackage[colorlinks = true,
linkcolor = blue,
urlcolor  = blue,
citecolor = blue,
anchorcolor = blue]{hyperref}
\usepackage{tikz}
\usetikzlibrary{matrix,arrows}
\renewcommand{\arraystretch}{1.2}
\usepackage{newtxtext}

\usepackage[format=plain,
            labelfont={bf,it},
            textfont=it]{caption}

\newcounter{mycount}
\theoremstyle{plain}
\newtheorem{theorem}[mycount]{Theorem}

\newtheorem{lemma}[mycount]{Lemma}
\newtheorem{proposition}[mycount]{Proposition}

\newtheorem{remark}{Remark}
\theoremstyle{definition}
\newtheorem{definition}{Definition}

\theoremstyle{example}

\theoremstyle{remark}
\numberwithin{equation}{section}
\numberwithin{figure}{section}

\def\des{\textsf{des}}

\def\iasc{\textsf{iasc}}
\newcommand{\asc}{\textsf{asc}}
\def\zero{\textsf{zero}}
\def\rep{\textsf{rep}}
\def\rmin{\textsf{rmin}}
\def\lmax{\textsf{lmax}}
\def\lmin{\textsf{lmin}}
\def\rmax{\textsf{rmax}}
\def\AND{\quad\text{and}\quad}
\def\WITH{\quad\text{with}\quad}
\DeclareRobustCommand{\Stirling}{\genfrac\{\}{0pt}{}}
\DeclareMathOperator{\Max}{\textnormal{max}}
\DeclareMathOperator{\Log}{\textnormal{log}}
\DeclareMathOperator{\Exp}{\textnormal{exp}}

\newcommand{\dd}{\mathop{}\!\textnormal{d}}
\def\le{\leqslant}
\def\ge{\geqslant}
\def\tr#1{\lfloor #1\rfloor}

\def\lpa#1{\bigl({#1}\bigr)}
\def\Lpa#1{\Bigl({#1}\Bigr)}

\def\llpa#1{\biggl({#1}\biggr)}
\def\ve{\varepsilon}

\newcommand{\SEPATTERN}{
	\draw[step=1, xshift=14pt, yshift=14pt, \cfill, line cap=round] (0,0) grid (3,3);
	\draw[step=1, xshift=14pt, yshift=14pt, thick] (0,1) -- (3,1);
	\draw[step=1, xshift=14pt, yshift=14pt, thick] (1,0) -- (1,3);
	\foreach \x/\y in {1/2,2/3,3/1} \node[disc, fill=black] at (\x,\y) {};
}

\newcommand{\sepattern}{\!\raisebox{-0.3em}{
		\begin{tikzpicture}[line width=0.7pt, scale=0.15]
		\tikzstyle{disc} = [circle,thin,draw=black, minimum size=1.7pt, inner sep=0pt ]
		\SEPATTERN
		\end{tikzpicture}}
}

\newcommand{\cfill}{black!40}

\graphicspath{{./fb_dim_figs/}}

\begin{document}

\title[Asymptotics and statistics on Fishburn Matrices
II]{Asymptotics and statistics on Fishburn Matrices: dimension
distribution and a conjecture of Stoimenow\thanks{This work was
partially supported by the FWF-MOST (Austrian-Taiwanese) Grant I
2309-N35.}}

\author{Hsien-Kuei Hwang, Emma Yu Jin and Michael J. Schlosser}
\thanks{The first author was also partially supported by an
Investigator Award from Academia Sinica under the Grant
AS-IA-104-M03. The second and the third author are partially
supported by the Austrian Research Fund FWF under the Grant P 32305.}

\address{Institute of Statistical Science, Academia Sinica, Taipei, 
115, Taiwan}\email{hkhwang@stat.sinica.edu.tw}

\address{Fakult\"{a}t f\"{u}r Mathematik, Universit\"{a}t Wien, 
Vienna, Austria} \email{yu.jin@univie.ac.at}
\email{michael.schlosser@univie.ac.at}

\maketitle

\begin{abstract}

We establish the asymptotic normality of the dimension of large-size
random Fishburn matrices by a complex-analytic approach. The
corresponding dual problem of size distribution under large dimension
is also addressed and follows a quadratic type normal limit law.
These results represent the first of their kind and solve two open
questions raised in the combinatorial literature. They are presented
in a general framework where the entries of the Fishburn matrices are not
limited to binary or nonnegative integers. The analytic saddle-point
approach we apply, based on a powerful transformation for $q$-series
due to Andrews and Jel\'\i nek, is also useful in solving a
conjecture of Stoimenow in Vassiliev invariants.
    
\end{abstract}

\tableofcontents

\section{Introduction and main results}
\label{S:intro}

Fishburn matrices (abbreviated as FMs), introduced by Peter Fishburn
in 1970 during his study of interval orders \cite{Fishburn1970}, are
upper-triangular square matrices with nonnegative integers as entries
such that no row and no column contains exclusively zeros. They also
appeared a few years later under a different guise in the study of
transitively directed graphs by Andresen and Kjeldsen
\cite{Andresen1976}, where essentially a recursive formula was given
on the number of \emph{primitive FMs} (FMs with entries $0$ or $1$)
with respect to the dimension and the first row sum (which is
$\xi(n,k)$ in \cite{Andresen1976}; see also \S~\ref{S:AK-1976}). For
example, all FMs with \emph{size} (or sum of all entries) equal to
$4$ are depicted in Figure~\ref{F:fishburn4} and all primitive FMs of
dimension $3$ in Figure \ref{F:fishburndim}.

\begin{figure}[!h]
	\begin{center}
			$$\begin{pmatrix*}[r]
			4\\
			\end{pmatrix*}
			\begin{pmatrix*}[r]
			1 \, 2 \\
			  \, 1
			\end{pmatrix*}\begin{pmatrix*}[r]
			2 \, 1 \\
			  \, 1
			\end{pmatrix*}\begin{pmatrix*}[r]
			1 \, 1 \\
			  \, 2
			\end{pmatrix*}\begin{pmatrix*}[r]
			2 \, 0 \\
			  \, 2
			\end{pmatrix*}\begin{pmatrix*}[r]
			3 \, 0 \\
			  \, 1
			\end{pmatrix*}\begin{pmatrix*}[r]
			1 \, 0\\
			  \, 3
			\end{pmatrix*}$$
			$$\begin{pmatrix*}[r]
			1 \, 1 \, 0\\
			  \, 1 \, 0\\
			  \,    \, 1
			\end{pmatrix*}\,\,\begin{pmatrix*}[r]
			1 \, 0 \, 1\\
			  \, 1 \, 0\\
			  \,   \, 1
			\end{pmatrix*}\,\,\begin{pmatrix*}[r]
			1 \, 0 \, 0\\
			  \, 1 \, 1\\
			  \,   \, 1
			\end{pmatrix*}\,\,\begin{pmatrix*}[r]
			2 \, 0 \, 0\\
			  \, 1 \, 0\\
			  \,   \, 1
			\end{pmatrix*}\,\,\begin{pmatrix*}[r]
			1 \, 0 \, 0\\
			  \, 2 \, 0\\
			  \,   \, 1
			\end{pmatrix*}\,\,\begin{pmatrix*}[r]
			1 \, 0 \, 0\\
			  \, 1 \, 0\\
			  \,   \, 2
			\end{pmatrix*}\,\,\begin{pmatrix*}[r]
			1 \, 1 \, 0\\
			  \, 0 \, 1\\
			  \,   \, 1
			\end{pmatrix*}\,\,\begin{pmatrix*}[r]
			1 \, 0 \, 0 \, 0\\
			  \, 1 \, 0 \, 0\\
			  \,   \, 1 \, 0\\
			  \,   \,   \, 1
	      \end{pmatrix*}$$
		\end{center}   
	\caption{All $15$ FMs of size $n=4$. The average 
    dimension of these matrices is 
    $\frac1{15}(1\cdot 1+2\cdot 6+3\cdot 7+4\cdot 1)\approx 2.533$, 
	which is already close to our asymptotic approximation 
    $\frac{6}{\pi^2}n=\frac{24}{\pi^2}\approx2.432$ in Theorem 
	\ref{T:dimdis}.}
	\label{F:fishburn4} 
\end{figure}
\begin{figure}[!h]
	\begin{center}
		$$\begin{pmatrix*}[r]
		1 \, 0 \, 0\\
		  \, 1 \, 0\\
		  \,   \, 1
		\end{pmatrix*}\,\begin{pmatrix*}[r]
		1 \, 1 \, 0\\
		  \, 1 \, 0\\
		  \,   \, 1
		\end{pmatrix*}\,\begin{pmatrix*}[r]
		1 \, 0 \, 0\\
		  \, 1 \, 1\\
		  \,   \, 1
		\end{pmatrix*}\,
		\begin{pmatrix*}[r]
		1 \, 0 \, 1\\
		  \, 1 \, 0\\
		  \,   \, 1
		\end{pmatrix*}\,
		\begin{pmatrix*}[r]
		1 \, 1 \, 0\\
		  \, 1 \, 1\\
		  \,   \, 1
		\end{pmatrix*}\,\begin{pmatrix*}[r]
		1 \, 1 \, 1\\
		  \, 1 \, 0\\
		  \,   \, 1
		\end{pmatrix*}\,\begin{pmatrix*}[r]
		1 \, 0 \, 1\\
		  \, 1 \, 1\\
		  \,   \, 1
		\end{pmatrix*}\,\begin{pmatrix*}[r]
		1 \, 1 \, 1\\
		  \, 1 \, 1\\
		  \,   \, 1
		\end{pmatrix*}\,\begin{pmatrix*}[r]
		1 \, 1 \, 0\\
		  \, 0 \, 1\\
		  \,   \, 1
		\end{pmatrix*}\,
		\begin{pmatrix*}[r]
		1 \, 1 \,  1\\
		  \, 0 \, 1\\
		  \,   \, 1
		\end{pmatrix*}\,\,$$
	\end{center}   
	\caption{All $10$ primitive FMs of dimension $n=3$. The average
    size (sum of entries) of these matrices is 
	$\frac1{10}(3\cdot 1+4\cdot 4+5\cdot 4+6\cdot 1)=4.5$ 
	while the asymptotic average size equals $\frac14n(n+1)=3$ in 
	Theorem \ref{T:dimdis2}.}
	\label{F:fishburndim} 
\end{figure}

Apart from the connection between primitive FMs and transitively
directed graphs, it is now known that FMs are in bijection with
interval orders, $(\textbf{2+2})$-free posets, ascent sequences,
certain pattern-avoiding permutations and regular linearized chord
diagrams (regular LCDs), etc.; see for instance \cite{Bousquet-Melou2010, Dukes2010,
Fishburn1985}.

The numbers of FMs of a given size are known as the \emph{Fishburn
numbers} (see \cite{Claesson2011} and
\cite[\href{https://oeis.org/A022493}{A022493}]{oeis2019}), which can
be computed by the Taylor coefficients of the generating function
\begin{align}\label{E:genfish}
    \sum_{k\ge0}\prod_{1\le j\le k}\lpa{1-(1-z)^j}
    =1+z+2z^2+5z^3+15z^4+53z^5+217z^6
    +\cdots.
\end{align}
This (formal) generating function was derived by Zagier
\cite{Zagier2001}, using a recursive formula found earlier by
Stoimenow \cite{Stoimenow1998} for the number of regular LCDs with a
given length; we postpone the exact definition of LCDs and regular
LCDs to Section \ref{S:sto}. Stoimenow also made in the same paper
\cite{Stoimenow1998} a conjecture concerning the asymptotic relation
between the Fishburn numbers and the number of connected regular LCDs
of size $n$, which will be addressed in more detail at the end of
this section.

Since the seminal work \cite{Bousquet-Melou2010} by Bousquet-M\'elou,
Claesson, Dukes and Kitaev, much attention has been drawn to the
refined enumeration of Fishburn structures with respect to various
classical statistics; see for instance \cite{Bousquet-Melou2010,
Dukes2010, Fu2020, Jelinek2012, Jelinek2015, Kitaev2011, Kitaev2017,
Levande2013}. Two types of statistics among all members of the
Fishburn family are Eulerian and Stirling statistics \cite{Fu2020}:
any statistic whose distribution over a member of the Fishburn family
equals the distribution of the dimension (resp.\ the first row sum)
on FMs is called an \emph{Eulerian} (resp.\ a \emph{Stirling})
statistic; see Table~\ref{T:stat} for a summary of the
equidistributed Eulerian and Stirling statistics on six Fishburn
structures.
\begin{table}[htpb]
	\centering
	\begin{tabular}{  c | c | c }
		 \multicolumn{1}{c}{Fishburn structures}
		 &\multicolumn{1}{c}{Eulerian statistics}
		 & \multicolumn{1}{c}{Stirling statistics}
		 \\ \hline
	     FMs & \text{dimension} -- $1$ 
         & \makecell{sum of the first row\\ (or the last column) \\ number of weakly northeast cells}\\
         \hline
	     (\textbf{2}$+$\textbf{2})-free posets & 
         magnitude -- $1$  & number of minimal elements\\
         \hline      
	     Ascent sequences & \asc, \rep & \zero, \textsf{max}, \rmin \\
	     \hline
	     (\textbf{2}$-$\textbf{1})-avoiding sequences & \rep 
		 & \textsf{max} \\
		 \hline 
	     $(\sepattern\,)$-avoiding permutations 
         & \des, \iasc  & \lmin, \lmax, \rmax \\
         \hline
	     \makecell{Regular linearized\\ chord diagrams}
		 & \makecell{length of the initial \\ run of openers}
         & \makecell{number of pairs of arcs\\ 
		 $(a,b),(c,d)$ such that\\ $a<b=c-1<d-1$} \\
	     \hline
	\end{tabular}
    \vspace{3mm}
	\caption{Equidistributed Eulerian and Stirling Statistics on 
	Fishburn structures: statistics in the second (resp.\ third) and 
	column are all equidistributed with each other; see 
    \cite{Bousquet-Melou2010, Dukes2010, Fu2020, Jelinek2012, 
	Levande2013} for precise definitions.} \label{T:stat}
\end{table}

While there is a large literature on the combinatorial aspects of
statistics over Fishburn structures, very few studies have been
conducted on asymptotic and stochastic properties concerning
structures of large size; see \cite{Bringmann2014,Zagier2001} and our previous paper \cite{Hwang2020}. Questions such as (see \cite{Jelinek2012})
``\emph{what is the expected dimension of a random FM of size $n$
when each of the size-$n$ FMs is chosen with the same probability?}''
and ``\emph{what is the expected size of a random FM when all FMs of
the same dimension are equally likely?}'' have remained open, and the
primary purpose of this paper is to answer these questions in a more
\emph{complete} (including the variance and the limiting
distribution) and more \emph{systematic} (covering a wide class of
generalized FMs) way.

In contrast to the Stirling statistics worked out in
\cite{Hwang2020}, which have typically \emph{logarithmic behaviors}
(logarithmic mean and logarithmic variance), the Eulerian statistics
studied in this paper, namely, dimension distribution with fixed
size, have asymptotically \emph{linear mean and linear variance} (the
corresponding dual statistic, size distribution of fixed dimension,
is \emph{quadratic}). Such a contrast is well known for statistics on
permutations, but has remained mostly elusive on Fishburn structures.
Whichever the case, the limiting distribution of any statistic is 
normal as long as the variance goes unbounded, although the proof 
technicalities differ.

Since an FM of a given size can be viewed as an integer partition
(but allowing $0$ as entries) arranged on an upper-triangular matrix,
there is yet a third class of Poisson statistics examined in detail
in \cite{Hwang2020}: the number of occurrences of the smallest
nonzero entry in the matrix. Similar to the classical integer
partitions where $1$ has a predominant frequency, the smallest
nonzero entry in FMs also appears almost everywhere. But different
from the exponential limit law of the occurrences of the smallest
part in random integer partitions, the smallest entry in FMs has its
occurrences following mostly (but not always) a Poisson limit law;
see \cite{Hwang2020}. This indicates an even
higher concentration of the smallest entry near its expected value in
the context of random FMs. Such a viewpoint will also be useful in
interpreting our asymptotic results in this paper.

The approach developed in \cite{Hwang2020} relies on a direct
two-stage saddle-point method that is applied to the generating
functions with a sum-of-product form, and is very powerful in that it
is not only applicable to the asymptotics of a wide class of concrete
examples, but also provides an effective means of understanding the
limit laws of Stirling statistics. In the present paper, we further
extend the same saddle-point approach to Eulerian statistics. This
extension is however not straightforward as a direct application
fails due to the violent fluctuations in summing the dominant terms,
similar to the summands on the left-hand side of \eqref{E:genfish}.
It turns out that a key property needed is a generalized Rogers-Fine
identity derived by Andrews and Jel\'\i nek in \cite{Andrews2014}.
Furthermore, an additional difficulty arises in handling the
uniformity in the extra parameter of the probability generating
function.

Given any multiset $\Lambda$ of nonnegative integers with the 
generating function 
\begin{align}\label{E:lambda}
	\Lambda(z)=1+\lambda_1 z+\lambda_2z^2+\cdots,
\end{align} 
a \emph{$\Lambda$-FM} is an FM with entry set $\Lambda$. The original
FMs correspond to the situation when all $\lambda_j$'s equal $1$, and
the primitive FMs to $\lambda_j = \delta_{j,1}$, $j\ge1$, the
Kronecker symbol. Although such a matrix formulation requires that
all the coefficients $\lambda_j$ be nonnegative integers, our proof
is independent of this restriction and the $\lambda_j$'s can indeed
be any nonnegative reals.

It is known that if $\Lambda(z)$ is analytic at $z=0$ with 
$\lambda_1>0$ then the number of $\Lambda$-FMs of size $n$ is given 
by (see \cite{Hwang2020})
\begin{align}\label{E:an-counts}
	a_n := [z^n]\sum_{k\ge0}\prod_{1\le j\le k}
	\lpa{1-\Lambda(z)^{-j}}
	= cn^{\frac12} (\lambda_1\mu)^n n! 
	\lpa{1+O\lpa{n^{-1}}},
\end{align}
where 
\begin{align}\label{E:c-rho}
	(c,\mu) := \llpa{\frac{12\sqrt{3}}{\pi^{5/2}}
	\, e^{\frac{\pi^2}6\lpa{\frac{\lambda_2}{\lambda_1^2}
	-\frac12}}, \frac{6}{\pi^2}}.
\end{align}
Here $[z^n]f(z)$ denotes the Taylor coefficient of $f(z)$. We see
that the dominant asymptotic order (neglecting the leading constant
$c$) depends crucially on $\lambda_1$, but not on any other
$\lambda_j$'s with $j\ge2$, showing roughly the pervasiveness of $1$
in a typical $\Lambda$-FM. On the other hand, the expression of $c$,
as well as the violent cancellations of terms when summing the Taylor
expansions of the finite products on the left-hand side of
\eqref{E:an-counts}, implicitly points to the difficulty of the
analysis involved; see \cite{Hwang2020} for more precise results.

\subsection{Dimension distribution of fixed-size FMs}

Define the bivariate generating function (see \cite{Fu2020, 
Jelinek2012}) 
\begin{align}\label{E:Fzv0}
    F(z,v) := \sum_{n\ge0}P_n(v) z^n
    = \sum_{k\ge0}\prod_{1\le j\le k}
    \llpa{1-\frac1{1+v(\Lambda(z)^j-1)}},
\end{align}
as an extension of the generating function in \eqref{E:an-counts}, 
where $P_n(v)$ is the generating polynomial of the dimension of 
size-$n$ $\Lambda$-FMs with $P_n(1)=a_n$. 

\begin{theorem}[Open problem of \cite{Bringmann2014}]\label{T:dimdis}
Assume that $\Lambda(z)$ is analytic at $z=0$ with $\lambda_1>0$ and 
that all $\Lambda$-FMs of size $n$ are equally likely to be selected. 
Then the dimension $X_n$ of a random matrix is asymptotically 
normally distributed with mean and variance both linear in $n$, 
namely,
\begin{align}\label{E:dimdis}
	\frac{X_n-\mu n}{\sigma\sqrt{n}}
	\xrightarrow{d} \mathscr{N}(0,1),	
    \WITH
	(\mu,\sigma^2)
    :=\llpa{\frac{6}{\pi^2},
    \frac{3(12-\pi^2)}{\pi^4}},
\end{align}
where the symbol $\xrightarrow{d}$ stands for convergence in 
distribution and $\mathscr{N}(0,1)$ the standard normal distribution. 
\end{theorem}

See Figure~\ref{F:b1} for three different graphic renderings of the
histograms of $X_n$ when $\Lambda = \mathbb{N}$. Note that $\sigma^2
=\mu^2-\frac12\mu$, and the coefficient pair $(\mu,\sigma^2)$ is to
some extent universal as we will also see its occurrences in other
classes of FMs (albeit in slightly different scales).

\begin{figure}[!ht]
\begin{center}
	\begin{tabular}{ccc}
		\includegraphics[height=4cm]{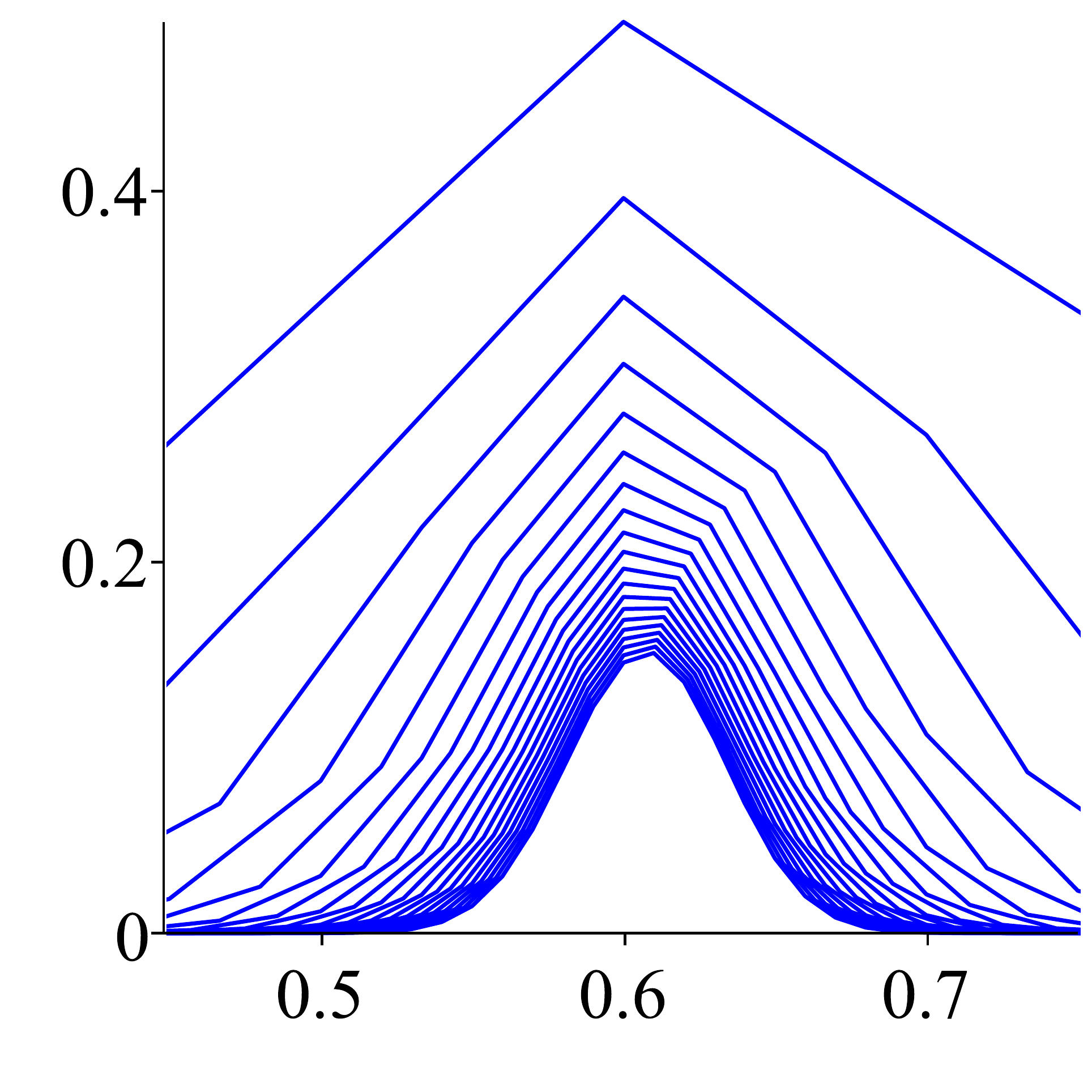} &
		\includegraphics[height=4cm]{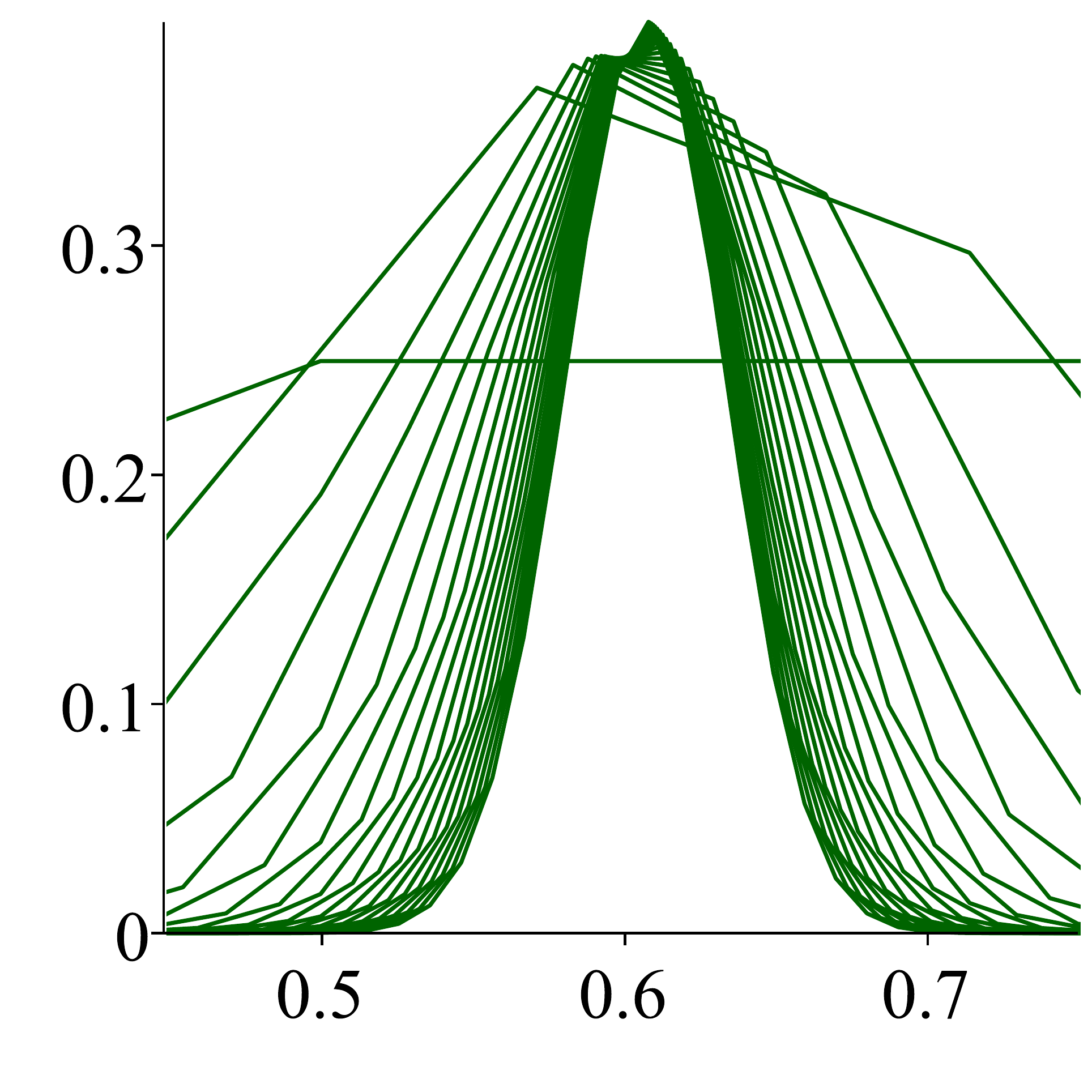} &
		\includegraphics[height=4cm]{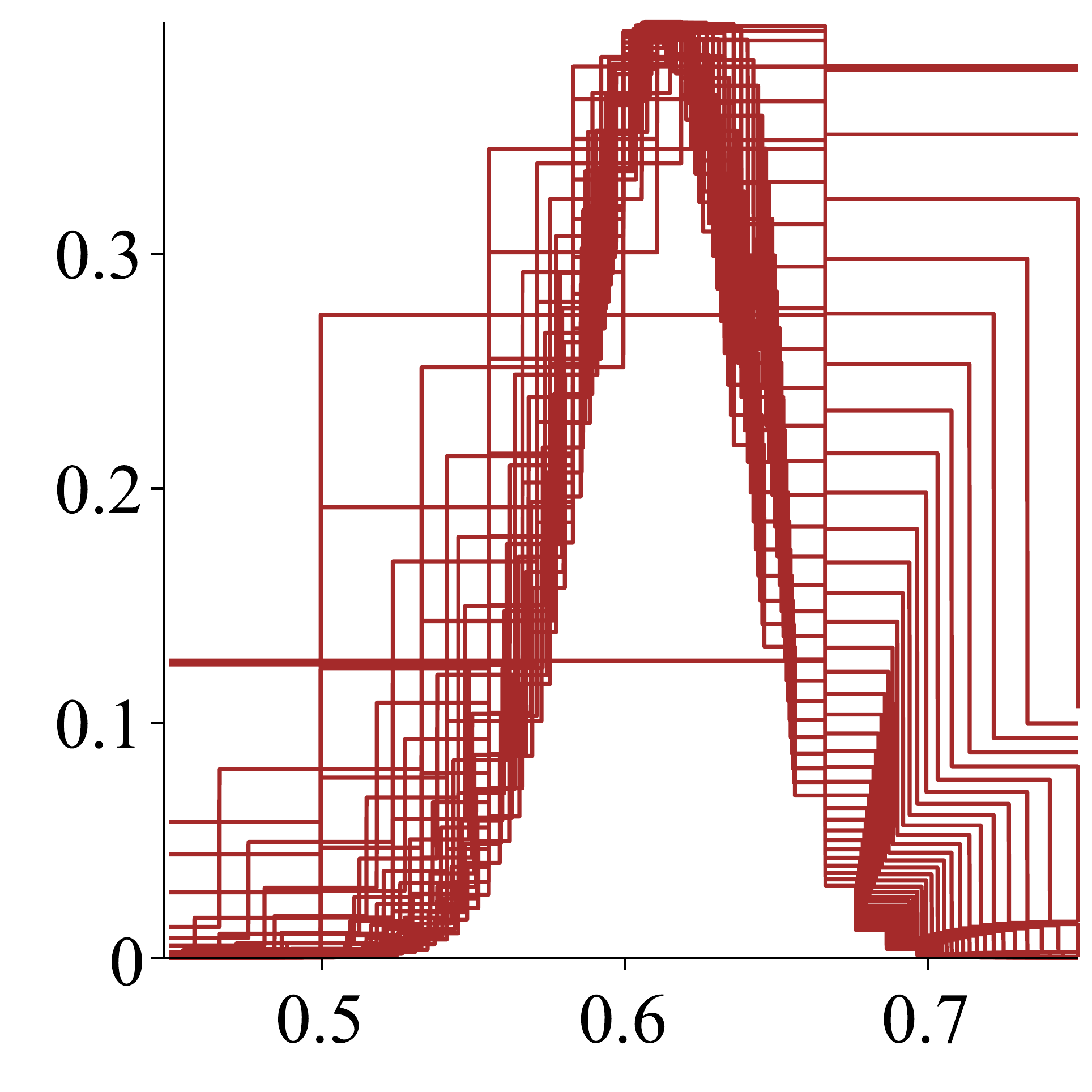} \\
		$\mathbb{P}(X_n=tn)$
		& $\sqrt{\mathbb{V}(X_n)}\,
		\mathbb{P}(X_n=t\mathbb{E}(X_n))$
		& $\sqrt{\mathbb{V}(X_n)}\,
		\mathbb{P}(X_n=\tr{t\mathbb{E}(X_n)})$
	\end{tabular}
\end{center}
\caption{Three different ways of visualizing the asymptotic normality
of $X_n$ where we plot the histograms of $X_n$ in the case when
$\Lambda(z) = (1-z)^{-1}$: for $n=5j$, $1\le j\le 20$ (left and
middle) and $n=3k$, $1\le k\le 33$ (right).} \label{F:b1}
\end{figure}

What is particularly remarkable here is that the central limit
theorem \eqref{E:dimdis} is \emph{independent} of $\Lambda$ (as long
as $\lambda_1>0$). The same also holds true for the first row sum
(see \cite{Hwang2020}), which behaves asymptotically like a normal
distribution with both mean and variance asymptotic to $\Log n$. Such
an ``invariance property'' \eqref{E:dimdis} may seem more surprising
than its logarithmic counterpart because linear statistics cover
stochastically a wider range of variations. We can view this
phenomenon from a few different angles.

First, from the asymptotic approximation \eqref{E:an-counts}, we see
that the number of general $\Lambda$-FMs with $\lambda_1>0$ behaves
roughly (modulo the leading constant $c$) like $\lambda_1^n$ times
the number of primitive FMs of the same size with $\Lambda(z)=1+z$.
So we next examine more closely how the magic constant $\mu$ appears
in random primitive FMs. The number of primitive FMs of size $n$ is
given by (see
\cite[\href{https://oeis.org/A138265}{A138265}]{oeis2019})
\[
    (a_n)_{n\ge1} 
    = (1, 1, 2, 5, 16, 61, 271, 1372, 7795, 49093, 339386, 
    2554596,\dots),
\]
and it turns out that in this special case, we have an unexpected 
\emph{identity} for the expected dimension:
\begin{align}\label{E:EXn-pfm}
    \mu_n 
    := \mathbb{E}(X_n) 
    = \frac{a_{n+1}}{a_n}\qquad(n\ge1);
\end{align}
see Section~\ref{S:mu-var} for a more general form as well as a
combinatorial proof of \eqref{E:EXn-pfm}; in other words,
\emph{the sum of the dimensions of all size-$n$ primitive FMs
matrices equals the number of size-$(n+1)$ primitive FMs}. In view of
\eqref{E:an-counts} and \eqref{E:EXn-pfm}, we immediately get the
asymptotic linearity of $\mathbb{E}(X_n)$ with the mean constant
$\mu$. In a similar manner, the second moment (and then the variance 
$\sigma^2$) can be approached via the same analytic and combinatorial 
arguments:
\[
    \sum_{1\le k\le n}\binom{k+1}2p_{n,k}
    +\sum_{1\le k\le n+1}k^2p_{n+1,k}
    = a_{n+3},
\]
where $p_{n,k}$ denotes the number of primitive FMs of size $n$ and
dimension $k$, which is $[v^k]P_n(v)$ from \eqref {E:Fzv0} when
$\Lambda(z)=1+z$, appearing also in
\cite[\href{https://oeis.org/A137252}{A137252}]{oeis2019}.

In addition, we will also derive finer asymptotic approximations for
$\mathbb{E}(X_n)$ and $\mathbb{V}(X_n)$.

\begin{theorem} \label{T:mu-var}
The mean and the variance of the dimension $X_n$ (defined in 
Theorem~\ref{T:dimdis}) of a random $\Lambda$-FM of size $n$ satisfy
\begin{align}\label{E:mu}
    \mathbb{E}(X_n) &= \mu \Lpa{n +\frac32}
	-\frac{\lambda_2}{\lambda_1^2}+ O\lpa{n^{-1}},\\
    \mathbb{V}(X_n) &= \sigma^2 \Lpa{n+\frac32}
    -\frac14+\frac{\lambda_2}{2\lambda_1^2}
    +O\lpa{n^{-1}}, \label{E:var}
\end{align}
where $(\mu,\sigma^2)$ is given in \eqref{E:dimdis}.
\end{theorem}
Note that the dependence of $\mathbb{E}(X_n)$ and $\mathbb{V}(X_n)$ 
on $\Lambda$ is weak: only the ratio of $\lambda_2$ and $\lambda_1^2$ 
appears in the constant terms, and similarly for higher central 
moments. For example, 
\begin{align*}
    \mathbb{E}\lpa{X_n-\mu_n}^3
    &= \frac{\pi^4-54\pi^2+432}{\pi^6}\, \Lpa{n+\frac32}
    + \frac1{12}
    -\frac{\lambda_2}{6\lambda_1^2}
    +O\lpa{n^{-1}},\\
    \mathbb{E}\lpa{X_n-\mu_n}^4
    &= 3\mathbb{V}(X_n)^2
    +\Lpa{6\sigma^4-\frac{\mu^2}{12}}n+O(1).
\end{align*}    
In principle, such calculations can be carried out further for all 
higher central moments and lead possibly to an alternative proof of 
\eqref{E:dimdis} by the method of moments. But the cancellations 
involved in such a process are very heavy and complex, so we will 
instead work out an analytic, cancellation-free approach. Other 
$\lambda_j$'s will appear in lower-order terms.

Interestingly, the source of the seemingly strange but omnipresent
ratio ``$\frac{\lambda_2}{\lambda_1^2}$'' in the second-order terms
will be indicated in Section \ref{ss:source}.

The same types of normal limit results are expected to hold for other
classes of FMs, and we will briefly examine two of them:
self-dual $\Lambda$-FMs (or persymmetric, namely, symmetric
with respect to the anti-diagonal), and $\Lambda$-FMs whose smallest
nonzero entries are $2$. The corresponding central limit theorems are
summarized in the Table~\ref{T:tab2}; see Section~\ref{S:final} for
more information.

\begin{table}[!ht]
\begin{center}
\begin{tabular}{ccc}
\makecell{$\Lambda$-FMs with $\lambda_1>0$} & 
Self-dual $\Lambda$-FMs with $\lambda_1>0$& 
\makecell{$\Lambda$-FMs with $\lambda_1=0$, $\lambda_2>0$}\\ 
(Theorem~\ref{T:dimdis}) &
(Theorem~\ref{T:dimdis3}) &
(Theorem~\ref{T:dimdis4}) \\ \hline
$\mathscr{N}(\mu n, \sigma^2 n)$ &
$\mathscr{N}(\mu n, 2\sigma^2 n)$ & 
$\mathscr{N}\lpa{\frac12\mu n, \frac12\sigma^2 n}$ \\
\end{tabular}    
\end{center}
\caption{A summary of the central limit theorems for the dimension of 
different types of random $\Lambda$-FMs. Note specially the change in 
the mean and the variance coefficients: while the halving in the last 
column is well expected, the asymptotic doubling of the variance in 
the self-dual FMs comes as a little surprise.}\label{T:tab2}
\end{table}

\subsection{Size distribution of fixed-dimension FMs}

We now address a dual problem: the size distribution of random
$\Lambda$-FMs with the same dimension. The problem is well-defined
when $\Lambda$ is finite and all coefficients of $\Lambda(z)$ are 
positive integers.

\begin{theorem}[Extended open problem 5.5 of \cite{Jelinek2012}]
\label{T:dimdis2}
Assume that $\Lambda(z)$ is a polynomial with positive coefficients 
and $\Lambda(1)\ne 1$, and that all $\Lambda$-FMs of dimension $m$ 
are equally likely. Then the size $Y_m$ of a random matrix is 
asymptotically normally distributed with mean and variance both of 
order $\Theta(m^2)$:
\begin{align}\label{E:dimdis2} 
    \frac{Y_m - \hat\mu m^2}{\hat\sigma m}
    \xrightarrow{d} \mathscr{N}(0,1),
\end{align}
where $\hat \mu,{\hat \sigma}^2>0$ are given by 
\begin{align}\label{E:mup-sigmap}
    (\hat \mu,{\hat \sigma}^2) 
    := \Lpa{\frac{\Lambda'(1)}{2\Lambda(1)},
    \frac12\Lpa{\frac{\Lambda'(1)+\Lambda''(1)}{\Lambda(1)}
    -\Lpa{\frac{\Lambda'(1)}{\Lambda(1)}}^2}}.
\end{align}
\end{theorem}
See Figure~\ref{F:b2} for three different plots of the histograms of
$Y_m$ in the case of binary FMs for which $(\hat \mu,{\hat \sigma}^2)
= (\frac14,\frac18)$. Note that, if $\Lambda(z) = 1+\sum_{1\le j\le
\ell}\lambda_j z^j$ with $\ell\ge1$ is a positive polynomial, then 
\[ 2\hat\sigma^2 =
\frac1{\Lambda(1)} \sum_{1\le j\le
\ell}\Lpa{j-\frac{\Lambda'(1)}{\Lambda(1)}}^2 \lambda_j>0.\]

\begin{figure}[!ht]
\begin{center}
	\begin{tabular}{ccc}
		\includegraphics[height=4cm]{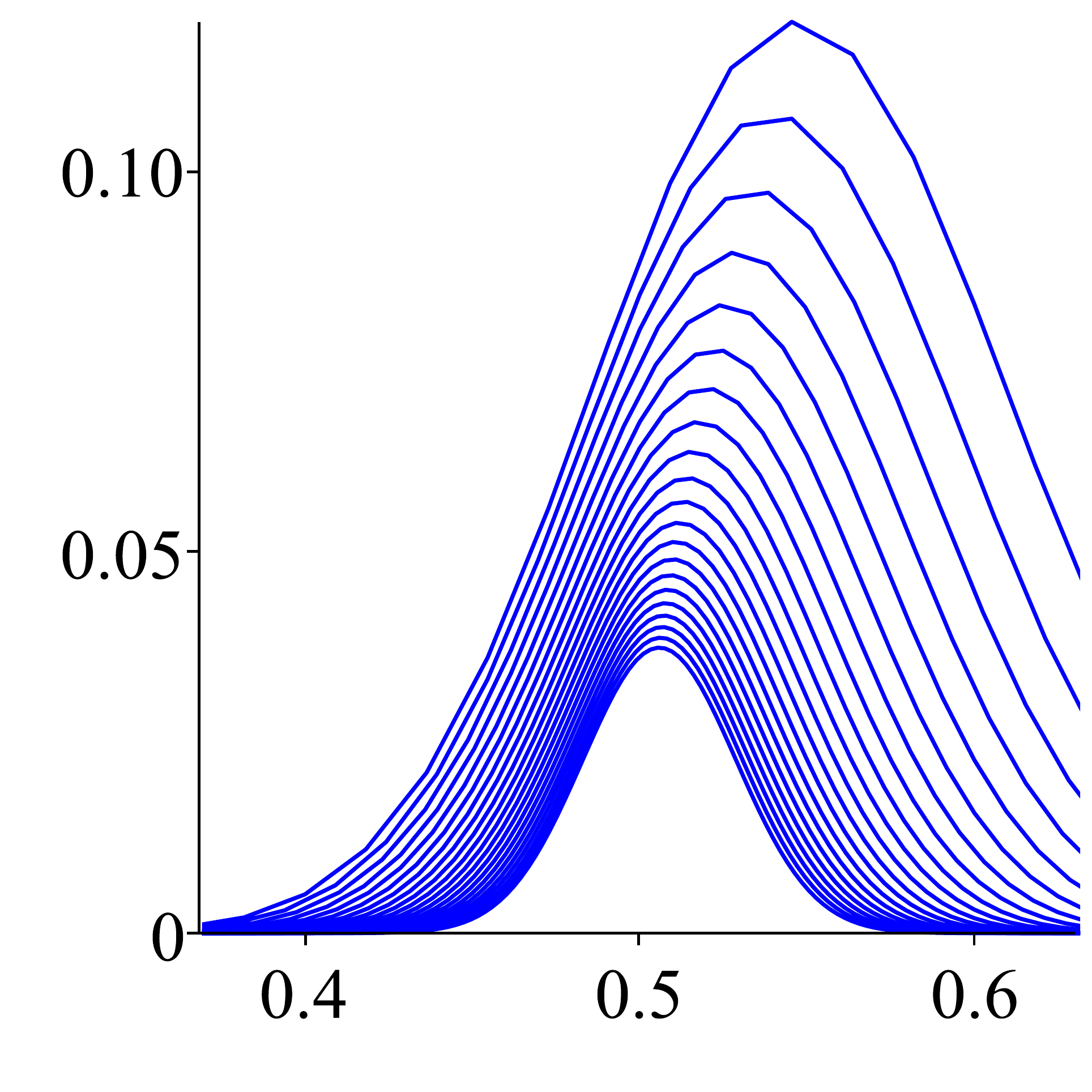} &
		\includegraphics[height=4cm]{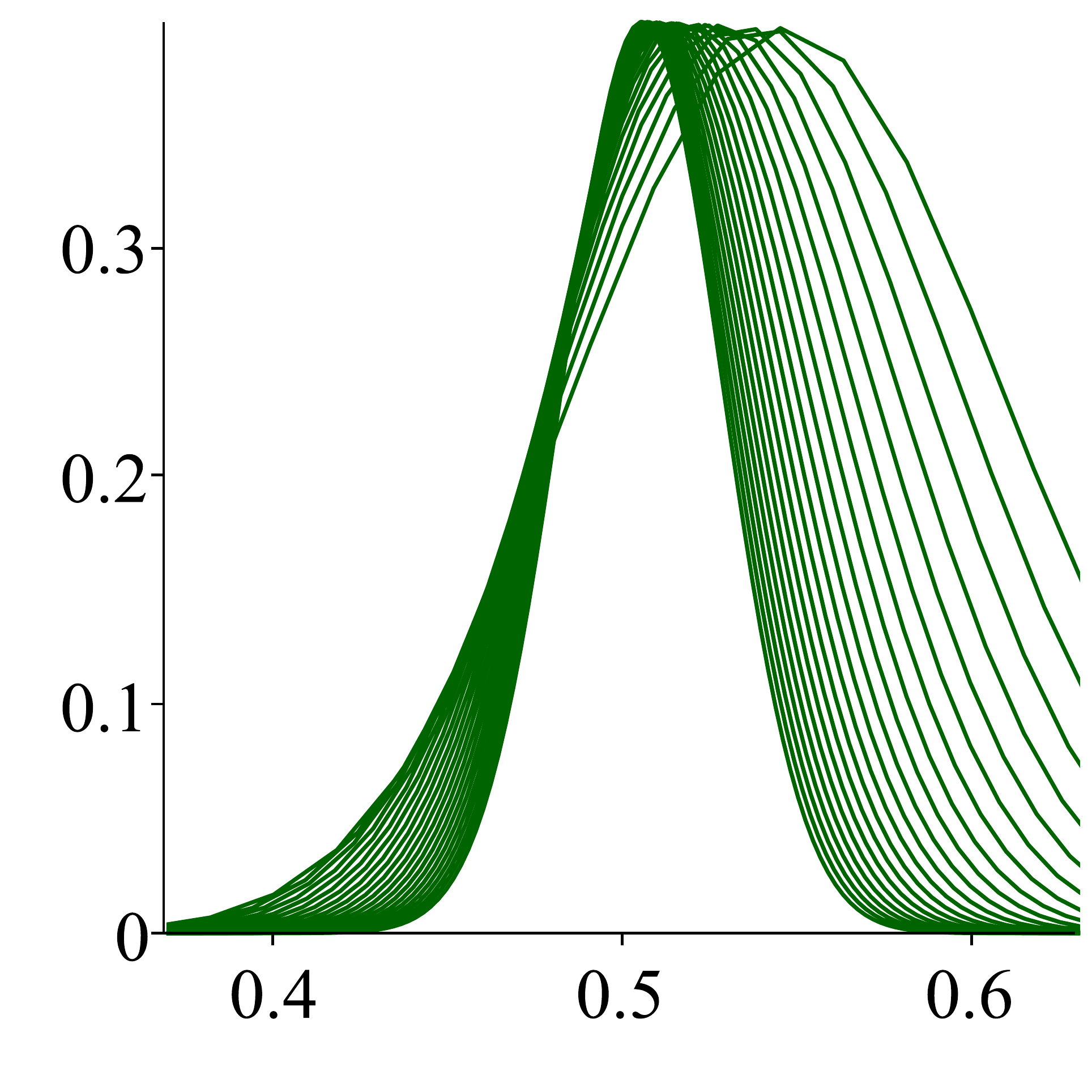} &
		\includegraphics[height=4cm]{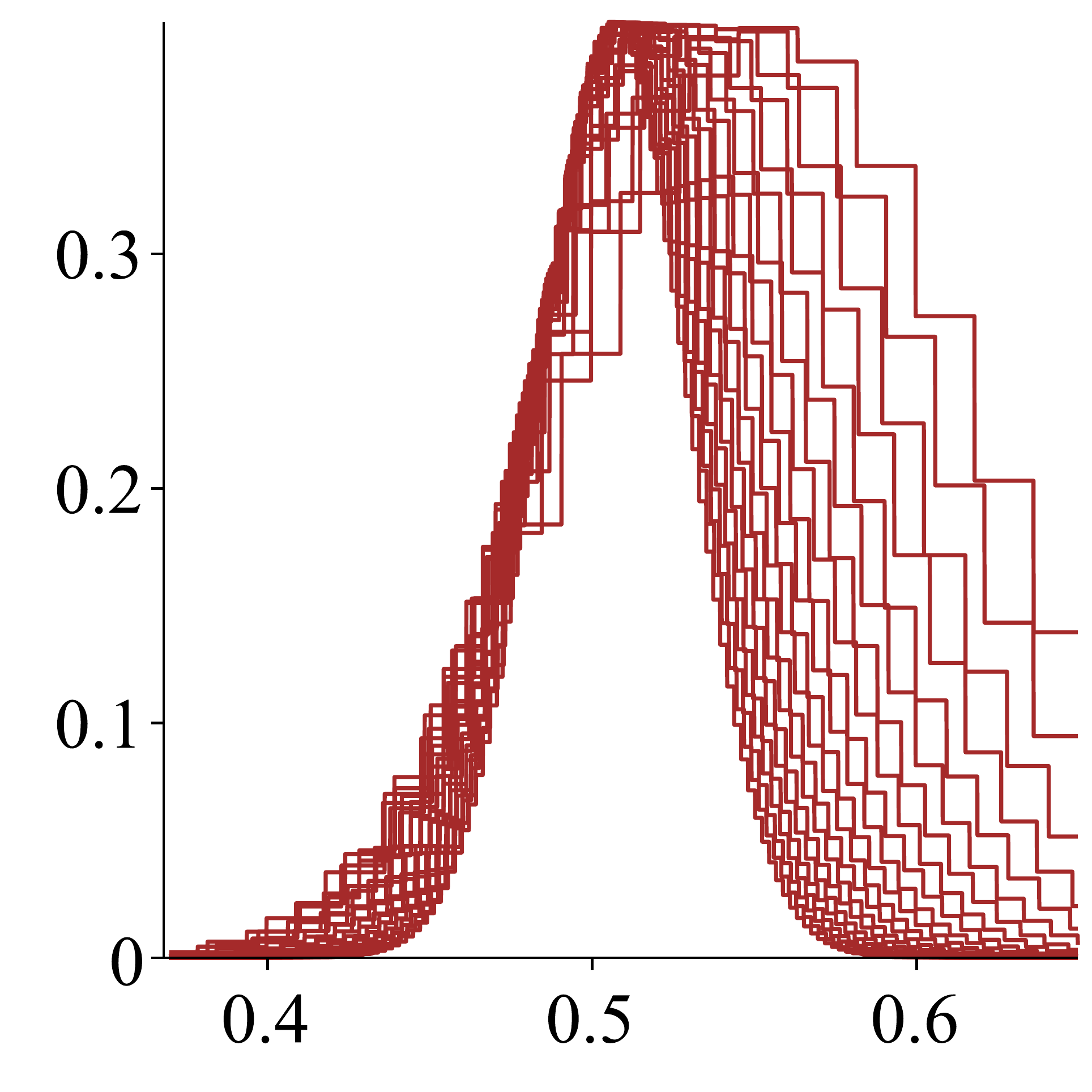} \\
		$\mathbb{P}(X_n=tn)$
		& $\sqrt{\mathbb{V}(X_n)}\,
		\mathbb{P}(X_n=t\mathbb{E}(X_n))$
		& $\sqrt{\mathbb{V}(X_n)}\,
		\mathbb{P}(X_n=\tr{t\mathbb{E}(X_n)})$
	\end{tabular}
\end{center}
\caption{Three different ways of visualizing the asymptotic normality 
of $Y_m$ where the histograms of $Y_m$ are given in the case when 
$\Lambda(z) = 1+z$: for $10\le n\le 30$.} \label{F:b2}
\end{figure}

\begin{remark}
Define the random variable $Y$ by $\mathbb{E}(z^{Y})
=\frac{\Lambda(z)}{\Lambda(1)}$. The quadratic behavior of $Y_m$
naturally suggests the question: ``\emph{what is the probability that
a randomly generated upper triangular matrix of dimension $m$ is
Fishburn when each entry is independently and identically distributed
as $Y$ (except for the upper-left and lower-right corners)?}'' Our
result implies particularly (see \eqref{E:m-dim-fm}) that in the
primitive case (when $Y$ is Bernoulli with mean $\frac12$), the
probability is asymptotic to
\[
	4\sum_{k\ge0}(-1)^{k}2^{-\binom{k+1}2}
	\sum_{0\le j\le k}\prod_{1\le \ell\le j}
	\frac1{1-2^{-\ell}}
	\approx 0.33359\dots.
\]
In other words, \emph{if we fix the two corners on the diagonal of
the matrix to be $1$, and generate all other entries by throwing an
unbiased coin, consistently putting $0$ or $1$ as the entry according
as the coin being head or tail, each independently of the others,
then \emph{more than one third of such matrices are Fishburn}.}
\end{remark}

\subsection{Asymptotic density of connected regular LCDs}

The proof of Theorem \ref{T:dimdis} is based on the saddle-point
approach developed by the first two authors in \cite{Hwang2020} and a
generalization of the Rogers-Fine identity due to Andrews and
Jel\'inek \cite{Andrews2014}, while Theorem \ref{T:dimdis2} follows
from a partial fraction decomposition and is simpler in nature.

It turns out that our saddle-point method is also useful in solving a
conjecture of Stoimenow \cite{Stoimenow1998} that was subsequently
reformulated by Zagier \cite{Zagier2001}, where the enumeration of
chord diagrams was studied in order to derive an upper bound for the
dimension of the Vassiliev invariants space for knots. Based on
numerical evidence, an asymptotic relation for the proportion of
connected regular LCDs (among all regular LCDs) was then conjectured;
see also \cite{Bringmann2014, Zagier2001}.

\begin{theorem}[A conjecture in \cite{Stoimenow1998}]\label{T:genzag}
Let $f_n$ be the number of regular LCDs of size $n$ (which equals the
$n$-th Fishburn number), and $g_n$ be that of connected regular LCDs
of size $n$. Then
\begin{align}\label{E:stoi-conj}
	\frac{g_n}{f_n}=e^{-1}\lpa{1+O\lpa{n^{-1}}}.
\end{align}
\end{theorem}
The same limit result also holds for the derangement probability and 
the proportion of connected (ordinary) chord diagrams, a well-known 
result; see for example \cite{Borinsky2018,Stein1978} and 
\cite[\href{https://oeis.org/A068985}{A068985}]{oeis2019}.

Let $g(z) := \sum_{n\ge1}g_n z^n$. Then the first few terms of $g(z)$ 
are given by (see 
\cite[\href{https://oeis.org/A022494}{A022494}]{oeis2019})
\begin{align}\label{E:gz-taylor}
    g(z)
    =z+z^2+2z^3+5z^4+16z^5+63z^6+293z^7+1561z^8
    +9321z^9+\cdots.
\end{align}
Our proof of \eqref{E:stoi-conj} relies crucially on the functional 
equation obtained by Zagier in \cite{Zagier2001}: 
\begin{align}\label{E:phi}
    \Phi(z, g(z))=1,\WITH
    \Phi(z,v)
    :=\frac{1}{1+v}\sum_{k\ge0}\prod_{1\le j\le k}
    \frac{1-(1-z)^j}{1+v(1-z)^j},
\end{align}
together with a generalized Rogers-Fine identity derived by
Andrews and Jel\'\i nek \cite{Andrews2014}. The function $\Phi$ is 
connected to $F$ in \eqref{E:Fzv0} when $\Lambda(z)=(1-z)^{-1}$ by 
\begin{align}\label{E:fphi}
    F(z,v)
    =\frac{1}{v}\Phi\Lpa{z,\frac{1}{v}-1}.
\end{align}
It is through this connection that our analytic techniques can be 
applied to solve the conjecture \eqref{E:stoi-conj}.

This paper is organized as follows. We prove in the next section
Theorem~\ref{S:mu-var} concerning the asymptotics of the expected
dimension and the variance. We also sketch briefly the approach we
developed in \cite{Hwang2020}. Then the normal limit law of the
dimension (Theorem \ref{T:dimdis}) is established in
Section~\ref{S:euler}, and the corresponding dual version in
Section~\ref{S:je}. Stoimenow's conjecture, which is now our
Theorem~\ref{T:genzag}, is confirmed in Section~\ref{S:sto}. Finally,
we describe very briefly in Section \ref{S:final} the limit results
for the dimension in the self-dual case, and the case when
$\lambda_1=0$, $\lambda_2>0$ and there exists at least one odd number
in the entry-set. We conclude by mentioning other possible
approximation theorems (convergence rates in the central theorems and
local limit theorems).

\emph{Throughout this paper, the generic symbols $c, \ve>0$ always
denote a constant and small quantity, respectively, whose values may
not be the same at each occurrence. In contrast, the pair
$(\mu,\sigma^2)$ always stands for the same value given in
\eqref{E:dimdis}. Furthermore, the notation $a_n\asymp b_n$ means
that the ratio $a_n/b_n$ remains bounded and not equal to zero as $n$ tends to
infinity.}

\section{The mean and the variance of the dimension}
\label{S:mu-var}

We prove Theorem~\ref{T:mu-var} in this section, together with a few 
related properties. 

\subsection{The generating functions of moments}

Define 
\begin{align}\label{E:mgf}
    M_h(z) := \sum_{n\ge0}a_n\mathbb{E}(X_n^h)z^n
    = \partial_s^h F(z,e^s)|_{s=0}\qquad(h=0,1,\dots)
\end{align}
to be (up to the normalizing factor $a_n$) the generating function of 
the $h$th moment of $X_n$, where $F$ is given in \eqref{E:Fzv0}. In 
particular, $M_0(z)$ corresponds to the generating function in 
\eqref{E:an-counts}.

\begin{lemma} The generating function of the $h$th moment of $X_n$ 
satisfies  
\begin{align}\label{E:Mhz}
    M_h(z) = U_h(z) M_0(z) + V_h(z),
\end{align}
for $h\ge0$, where ($\Stirling{h}{j}$ are the Stirling numbers of 
the second kind)
\begin{align}\label{Uhz}
    U_h(z) &:= \sum_{0\le \ell \le h}
    \Stirling{h+1}{\ell+1}(-1)^{h-\ell}
    \ell!\prod_{1\le j\le\ell}
    \frac{1}{1-\Lambda(z)^{-j}} \\
    V_h(z) &:= \sum_{0\le \ell \le h}
    \Stirling{h+1}{\ell+1}(-1)^{h+1-\ell}
    \ell!\sum_{0\le k<\ell}\prod_{k< j\le\ell}
    \frac1{1-\Lambda(z)^{-j}}.\nonumber
\end{align}
\end{lemma}
\begin{proof}
By taking the derivative with respect to $v$ on both sides of 
\eqref{E:Fzv0} and then substituting $v=1$, we obtain 
\begin{align*}
    M_1(z) &= \partial_vF(z,v)\bigr|_{v=1}
    = \sum_{k\ge0}\llpa{\sum_{1\le l\le k}\Lambda(z)^{-l}}
	\llpa{\prod_{1\le j\le k}\lpa{1-\Lambda(z)^{-j}}}\\
	&= \sum_{k\ge0}
	\llpa{-1+\frac{1-\Lambda(z)^{-k-1}}{1-\Lambda(z)^{-1}}}
	\llpa{\prod_{1\le j\le k}\lpa{1-\Lambda(z)^{-j}}}.
\end{align*}  
It follows that 
\begin{align}\label{E:M1}
	M_1(z) = \frac{M_0(z)-\Lambda(z)}
    {\Lambda(z)-1}.
\end{align}
In a similar way, 
\begin{align*}
    M_2(z) &=\lpa{\partial_v^2F(z,v)+\partial_vF(z,v)}\bigr|_{v=1}\\
    &= \sum_{k\ge0}
	\llpa{1-3\frac{1-\Lambda(z)^{-k-1}}{1-\Lambda(z)^{-1}}
    +2\frac{\lpa{1-\Lambda(z)^{-k-1}}\lpa{1-\Lambda(z)^{-k-2}}}
    {\lpa{1-\Lambda(z)^{-1}}\lpa{1-\Lambda(z)^{-2}}}}
	\llpa{\prod_{1\le j\le k}\lpa{1-\Lambda(z)^{-j}}}.
\end{align*}
Thus 
\begin{align}\nonumber
    M_2(z) &= M_0(z) 
    -\frac{3\lpa{M_0(z) - 1}}{1-\Lambda(z)^{-1}}
    +\frac{2\lpa{M_0(z) - 2 +\Lambda(z)^{-1}}}
    {\lpa{1-\Lambda(z)^{-1}}\lpa{1-\Lambda(z)^{-2}}} \\
    &= \frac{1+2\Lambda(z)-\Lambda(z)^2}
    {\lpa{\Lambda(z)-1}\lpa{\Lambda(z)^2-1}}\,M_0(z)-
    \frac{\Lambda(z)\lpa{3-2\Lambda(z)+\Lambda(z)^2}}
    {\lpa{\Lambda(z)-1}\lpa{\Lambda(z)^2-1}}.
    \label{E:M2}
\end{align}
The general form \eqref{E:Mhz} is then proved by the same arguments 
and induction.
\end{proof}
%

\subsection{Combinatorial interpretations}

Recall that $a_n$ and $\mu_n$ are defined in \eqref{E:an-counts} and
(\ref{E:EXn-pfm}), respectively. From \eqref{E:M1}, we have the
identity
\[
	\sum_{1\le j<n}\lambda_j a_{n-j}\mu_{n-j}
	= a_n - \lambda_n\qquad(n\ge1),
\]
where $a_n\mu_n = [z^n]M_1(z)$. In particular, in the primitive case 
when $\Lambda(z)=1+z$, we have a surprisingly simple identity for  
the expected dimension:
\[
	\mu_n = \frac{a_{n+1}}{a_n},
\]
or, in words, \emph{the expected dimension equals the ratio between
the number of primitive FMs of size $n+1$ and that of size $n$.} 

Similarly, for the second moment, we have the identity
\[
    a_n\mathbb{E}\lpa{X_n^2}
    +2a_{n+1}\mathbb{E}\lpa{X_{n+1}^2}
    = 2a_{n+3}-a_{n+1}.
\]
These simple relations certainly demand for combinatorial 
interpretations, which are given in the following forms. 

\begin{proposition}\label{P:pnk}
Let $p_{n,k}$ denote the number of primitive FMs of size $n$ and 
dimension $k$. Then for $n\ge1$
\begin{align}\label{E:dimid} 
	a_{n+1}&=\sum_{1\le k\le n}kp_{n,k},\\
    a_{n+3} &=\sum_{1\le k\le n}\binom{k+1}2p_{n,k}
    +\sum_{1\le k\le n+1}k^2p_{n+1,k}.\label{E:dimid2} 
\end{align}
\end{proposition}

Among the diverse Fishburn structures, we find it simpler to
interpret \eqref{E:dimid} and \eqref{E:dimid2} in the language of
ascent sequences, listed in Table~\ref{T:stat}. We can then translate
the recursive construction on primitive ascent sequences into
primitive FMs via the bijection in \cite{Dukes2010}.

\begin{definition}[Ascent sequence]\label{D:asc}
Let  $\mathcal{I}_n$ be the set of inversion sequences of length $n$, 
namely,
\begin{align*}
    \mathcal{I}_n
    := \{s=(s_1,s_2,\ldots,s_n):0\le s_j<j, 1\le j\le n\},
\end{align*}
For any sequence $s\in\mathcal{I}_n$, let
\begin{align}\label{E:asc}
    \asc(s) &:=|\{1\le j<n: s_j<s_{j+1}\}|
\end{align}
be the number of {\bf asc}ents of $s$. An \emph{inversion sequence}
$s\in \mathcal{I}_n$ is an \emph{ascent sequence} if for all $2\le j
\le n$, $s_j$ satisfies
\begin{align*}
	s_j\le \asc(s_1,s_2,\ldots, s_{j-1})+1.
\end{align*} 
An ascent sequence is \emph{primitive} if no consecutive entries are 
identical.
\end{definition}

\begin{proof} (Proposition~\ref{P:pnk})
It is known (see \cite{Bousquet-Melou2010, Dukes2010}) that $p_{n,k}$
also enumerates the number of primitive ascent sequences with $k-1$
ascents. For instance, $p_{4,3}=4$: the corresponding primitive 
ascent sequences are $0121$, $0120$, $0102$ and $0101$ and they are
in bijection with the following primitive FMs from left to right, 
respectively.
\[
    \begin{pmatrix*}[r]
	1 \, 0 \, 0\\
	  \, 1 \, 1\\
	  \,   \, 1
	\end{pmatrix*}\;
    \begin{pmatrix*}[r]
	1 \, 0 \, 1\\
	  \, 1 \, 0\\
	  \,   \, 1
	\end{pmatrix*}\;
    \begin{pmatrix*}[r]
	1 \, 1 \, 0\\
	  \, 1 \, 0\\
	  \,   \, 1
	\end{pmatrix*}\;
    \begin{pmatrix*}[r]
	1 \, 1 \, 0\\
	  \, 0 \, 1\\
	  \,   \, 1
	\end{pmatrix*}
\]

Given a primitive ascent sequence $s$ of length $n$ and with $k-1$
ascents, we add a new entry at the end of $s$, which can be any
integer from $[0,k]$ but not equal to the last entry $s_n$ of $s$. In
other words, there are $k$ possible ways to add such an integer so
that the resulting sequence is a primitive ascent sequence of length
$n+1$, which leads to (\ref{E:dimid}).

Now we extend the same proof to show \eqref{E:dimid2}. Given a
primitive ascent sequence $s$ of length $n+1$ and with $k-1$
ascents, we add two entries $x, y$ at the end of $s$, where $0\le x\le
k$, $x\ne s_{n+1}$, and $0\le y\le k$, $y\ne x$. That is, there are
$k^2$ possible values for the pair $(x,y)$. By Definition
\ref{D:asc}, the resulting sequence $s^*=s_1\ldots s_{n+1} xy$ is a
primitive ascent sequence of length $n+3$ such that if the 
penultimate entry is removed, the resulting sequence is still an 
ascent sequence. 

On the other hand, given a primitive ascent sequence $s$ of length
$n$ and with $k-1$ ascents, we add three entries $x, y, z$ at the end
of $s$ so that the resulting sequence $s^*=s_1\ldots s_n xyz$ is not
an ascent sequence if the penultimate entry is removed, i.e., $s^*$
satisfies $x<y<z=\asc(s^*)$. If $y=k+1$, then $s_n<x\le k$ and
$z=k+2$, implying that there are $k-s_n$ possible choices for $x$;
otherwise $0\le y\le k$. Since $0\le x<y\le k$ and $x\ne s_n$, there
are $\frac12k(k-1)+s_n$ different values for the pair $(x,y)$ and
$z=k+1$. It follows that there are in total
$\frac12k(k-1)+k=\frac12k(k+1)$ choices for $(x,y)$, and the
resulting sequence $s^*=s_1\ldots s_n xyz$ is a primitive ascent
sequence.

Since any primitive ascent sequence of length $n+3$ can be produced
by either construction, we thus conclude the identity
\eqref{E:dimid2}. 
\end{proof}

\begin{remark}
When $\Lambda(z)=(1-z)^{-1}$, we have instead the pair of relations
\[
\left\{
\begin{split}
    a_{n+1}-a_n&=\sum_{1\le k\le n}k\bar p_{n,k},\\
    a_{n+3}-a_{n+2}
    &=\sum_{1\le k\le n+1}k(k+2)\bar p_{n+1,k}
    -\sum_{1\le k\le n}\binom{k+1}{2}\bar p_{n,k}.
\end{split}\right.
\]
Such $\bar p_{n,k}$ denotes the number of size-$n$ FMs of dimension 
$k$; see \cite[\href{https://oeis.org/A137251}{A137251}]{oeis2019}. 
Similar combinatorial interpretations can be given as in the 
primitive case. 
\end{remark}

\subsection{The two-stage saddle-point approach}\label{S:2stage}

For self-containedness and to pave the way for proving the asymptotic
normality of the dimension, we sketch here the major steps of the
two-stage saddle-point method developed in \cite{Hwang2020} for
\eqref{E:an-counts}, at the same time also indicating how to
obtain a finer asymptotic expansion for $a_n$.

\subsubsection{$q$-series transformation}

The approach starts from the generating function \eqref{E:an-counts},
which contains nevertheless terms with negative coefficients in the
Taylor expansion of $1-\Lambda(z)^{-j}$, which in turn, after
multiplication over $1\le j\le k$, results in alternating terms that
produce severe cancellations in the final summation; see
\cite{Hwang2020} for more details. Instead of manipulating the heavy
cancellations, it is technically more convenient to work on the
right-hand side of the identity
\begin{align}\label{E:aj-q}
    \sum_{k\ge0}\prod_{1\le j\le k}
    \lpa{1-\Lambda(z)^{-j}}
    = \sum_{k\ge0}\Lambda(z)^{k+1}\prod_{1\le j\le k}
    \lpa{\Lambda(z)^j-1}^2,
\end{align}
as the right-hand side of \eqref{E:aj-q} contains only terms with
nonnegative Taylor coefficients. This identity is obtained by 
applying a $q$-identity due to Andrews and Jel\'\i nek 
\cite[Proposition 2.3]{Andrews2014}: 
\begin{equation}\label{E:Andrews}
\begin{split}
	&\sum_{k\ge0}u^k\prod_{1\le j\le k}
	\left(1-\frac{1}{(1-s)(1-t)^{j-1}}\right)\\
	&\qquad=\sum_{k\ge0}(1-s)(1-t)^k
	\prod_{1\le j\le k}\lpa{\lpa{1-(1-s)(1-t)^{j-1})}
	\lpa{1-u(1-t)^j}},
\end{split}	
\end{equation}
after substituting $u=1$ and $s=t=1-\Lambda(z)$ on both sides. 

\subsubsection{The exponential prototype}
\label{S:tep}

From the transformed generating function \eqref{E:aj-q}, it proves
much simpler to work out first the special case when $\Lambda(z) =
e^z$ (because we assume $\lambda_1>0$). We will see later how to 
recover the asymptotics of $a_n$ in general cases. 

Let 
\begin{align}\label{E:Ekz}
    E(z) := \sum_{k\ge0} E_k(z),\WITH
	E_k(z) := e^{(k+1)z}\prod_{1\le j\le k}
    \lpa{e^{jz}-1}^2,
\end{align}
and  
\[
    e_n := \sum_{0\le k\le\tr{\frac12n}}e_{n,k},\WITH 
    e_{n,k} := [z^n]E_k(z)
    = \frac1{2\pi i}\oint_{|z|=r}z^{-n-1}E_k(z) \dd z,
\]
where $r>0$. The sequence $n!e_n$ is essentially \emph{the number of
labelled interval orders on $n$ points}; see \cite{Brightwell2011}
and \cite[\href{https://oeis.org/A079144}{A079144}]{oeis2019}.

\subsubsection{The saddle-point method}
\label{S:spm}

Define 
\begin{align}\label{E:Iz}
    I(z)
	:=\int_{0}^{z}\frac{t}{1-e^{-t}}\dd t
	= \frac{z^2}2+\text{dilog}(e^{-z}),
\end{align}
where $\text{dilog}(z)$ denotes the dilogarithm function
\[
    \text{dilog}(z) 
	:= \int_0^z \frac{\Log u}{1-u}\,\dd u.
\]	

The asymptotic analysis of $e_n$ is then split into the following 
steps. 

\begin{enumerate}[(i)]\itemsep9pt

\item Apply first the saddle-point bound for the Taylor coefficients 
$e_{n,k}\le r^{-n}E_k(r)$, where $r>0$ solves $nE_k(r)=rE_k'(r)$, 
namely,  
\[
	n-(k+1)r
	= \sum_{1\le j\le k}\frac{2jr}{1-e^{-jr}}
	\sim \frac{2I(kr)}{r}.
\]
By the asymptotic behaviors of $I(x)$ as $x\to0$ and $x\to\infty$, we 
see that such an $r$ exists as long as $0\le k<\tr{\frac12n}$ and 
satisfies $r\asymp (n-2k)(k+1)^{-2}$. 
This choice of $r$ 
then gives ($q=k/n$)
\begin{align}\label{E:enk-spb}
	e_{n,k} = O\lpa{n^{n+1}e^{\phi(q,\mu^{-1})n}},
	\WITH \phi(q,\xi) := 2q\Log(e^{q\xi}-1)-1-\Log\xi,
\end{align}
where $\xi$ is connected to $q$ by the relation $2I(q\xi) = \xi$.

\item Find the positive solution pair $(q,\xi)$ of the equations 
\[
    \partial_q\phi(q,\xi)=0 \,\,\mbox{ and }\,\, 2I(q\xi) = \xi,
\]
so as to maximize $\phi(q,\xi)$. The solution is then given by 
\begin{align}\label{E:qxr}
	(q,\xi) =\lpa{\mu\Log 2, \mu^{-1}}.
\end{align}

\item We then further shrink the dominant range to $|k-qn|\le
n^{\frac12+\ve}$, $\ve>0$, where most contribution to $e_n$ will
come. It suffices to choose $\ve=\frac18$, and show, by the
saddle-point bound \eqref{E:enk-spb} and the concavity of
$\phi(q,\xi)$, that the contribution to $e_n$ of $e_{n,k}$ from the
range $|k-q n|\ge n^{\frac58}$ is asymptotically negligible.

\item In the central range $|k-qn|\le n^{\frac 58}$, show that the 
integral 
\[
    \int_{\substack{z=re^{i\theta}\\
    n^{-\frac38}\le |\theta|\le \pi}} z^{-n-1}E_k(z) \dd z
\]
is asymptotically negligible. The key property used is the following
concentration inequality (see also \cite[Lemma 13]{Hwang2020})
\begin{align}\label{E:Ekz-o}
	|E_k(re^{it})|
	\le E_k(r) \Exp\Lpa{-\frac{(k+1)^2rt^2}{\pi^2}},
\end{align}
uniformly for $k\ge0$, $r>0$ and $|t|\le\pi$.

\item Then inside the ranges $|k-qn|\le n^{\frac58}$, compute the 
integral 
\[
    \int_{\substack{z=re^{i\theta}\\
    |\theta|\le n^{-\frac38}}} z^{-n-1}E_k(z) \dd z
\]
by more precise local expansions, standard Gaussian approximation, 
and term-by-term integration, after deriving a fine asymptotic 
expansion for the saddle-point $r$. 

\item Summing over the asymptotics of $e_{n,k}$ and approximating the 
sum by an integral give \eqref{E:an-counts}.

\item Refine steps (iv) and (v) by using a longer expansion if more 
terms in the asymptotic expansion are desired. 
\end{enumerate}

We then obtain not only \eqref{E:an-counts} when $\Lambda(z)=e^z$ but 
also a refined asymptotic expansion 
\begin{align}\label{E:bn-ae}
	\frac{e_n}{cn^{\frac12}\mu^nn!} 
	= 1+\sum_{1\le j<j_0}\tilde{d}_j n^{-j}+O\lpa{n^{-j_0}},
\end{align}
for any $j_0=1,2,\dots$, where $(c,\mu)$ is given in 
\eqref{E:c-rho} when $\Lambda(z)=e^z$, and, in particular,  
\begin{align*}\small
	\tilde{d}_1 = \frac38+\frac{\pi^2}{144},
    \quad \text{and} \quad \tilde{d}_2 = -\frac{7}{128}
	-\frac{\pi^2}{1152}
    +\frac{\pi^4}{41472}.
\end{align*}

\subsubsection{From $e^z$ back to $\Lambda(z)$}\label{ss:source}

To recover the asymptotics \eqref{E:an-counts} from the special case 
when $\Lambda(z)=e^z$, we use the following change-of-variables 
arguments based on the Cauchy integral representation of $a_n$:
\begin{align}\label{E:Akz}
	a_n &:= \sum_{0\le k\le\tr{\frac12n}}
	[z^n]A_k(z)\WITH
	A_k(z) := \Lambda(z)^{k+1}
	\prod_{1\le j\le k}\left(\Lambda(z)^j-1\right)^2,
\end{align}
We then make the change of variables $\Lambda(z)=e^y$, which is 
locally invertible when $z\sim0$ because $\lambda_1>0$, so that 
$z=\lambda_1^{-1} y\psi(y)$, where $\psi(y)$ satisfies 
\[
    \psi(y) = 1+\Lpa{\frac12-\frac{\lambda_2}{\lambda_1^2}}y
	+\Lpa{\frac16-\frac{\lambda_1^2\lambda_2+\lambda_1\lambda_3-
	2\lambda_2^2}{\lambda_1^4}}y^2+\cdots.
\]
The analyticity of $\Lambda$ also implies the boundedness of $\psi$ 
when $y$ is small. Here we also see the magic constant 
``$\frac{\lambda_2}{\lambda_1^2}$'' appears in the linear term, which 
is the source of all the occurrences in the second-order terms in the 
moments approximations; see \eqref{E:mu} and \eqref{E:var}. Then 
\begin{align*}
	a_n &= \lambda_1^n [y^n]\Psi_n(y)
	\sum_{0\le k\le\tr{\frac12n}} E_k(y)
	= \lambda_1^n [y^n]\Psi_n(y)
	\sum_{0\le k\le n}e_ky^k,	
\end{align*}
where $\Psi_n(y) := \psi(y)^{-n-1}(\psi(y)+y\psi'(y))$.
With $d := -\frac12+\frac{\lambda_2}{\lambda_1^2}$, we have the 
asymptotic expansion
\[
	\Psi_n\Lpa{\frac tn}
	= e^{dt}
	\Lpa{1+\frac{\varpi_1(t)}{n}+\frac{\varpi_2(t)}{n^2}
	+\cdots},
\]
where the $\varpi_j$'s are polynomials in $t$ of degree $2j$. Now
expand each term on the right-hand side at $t=t_0 := \frac{\pi^2}6$,
compute the coefficient of $t^n$ term by term, and then estimate the
corresponding error terms; see \cite{Hwang2020} for details. We then
obtain an asymptotic expansion in decreasing powers of $n$, which,
for easier reference, is stated formally as follows. All steps
involved are readily coded (except for the justification ones).

\begin{proposition} Assume that $\Lambda(z)$ is analytic at $z=0$ and 
$\lambda_1>0$. Then the number of $\Lambda$-FMs of size $n$ satisfies 
the asymptotic expansion 
\begin{align}\label{E:an-ae}
	\frac{a_n}{cn^{\frac12}(\lambda_1\mu)^nn!} 
	= 1+\sum_{1\le j<j_0}d_j n^{-j}+O\lpa{n^{-j_0}},
\end{align}
for any $j_0=1,2,\dots$, where $(c,\mu)$ is given in 
\eqref{E:c-rho}, and, in particular,  
\begin{align}\small \label{E:d1}
	d_1 &= \frac38+\frac{19\lambda_1^2-36\lambda_2}
	{144\lambda_1^2}\,\pi^2
	+\frac{\lambda_1^2+12\lambda_1\lambda_3-12\lambda_2^2}
	{432\lambda_1^4}\,\pi^4,\\
	d_2 &= -\frac{7}{128}
	-\frac{(19\lambda_1^2-36\lambda_2)\pi^2}
    {1152\lambda_1^2}
    -\frac{(35\lambda_1^4+456\lambda_1^2\lambda_2
	+1872\lambda_1\lambda_3-2304\lambda_2^2)\pi^4}
	{41472\lambda_1^4}\nonumber\\
	&\quad+\frac{(7\lambda_1^6-12\lambda_1^4\lambda_2
	+228\lambda_1^3\lambda_3-228\lambda_1^2\lambda_2^2
	+288\lambda_1^2\lambda_4-1008\lambda_1\lambda_2\lambda_3
	+720\lambda_2^3)\pi^6}{62208\lambda_1^6}\nonumber\\
	&\quad-\frac{(5\lambda_1^4-12\lambda_1^2\lambda_2
	+24\lambda_1\lambda_3-12\lambda_2^2) 
	(\lambda_1^4-12\lambda_2\lambda_1^2
	-24\lambda_1\lambda_3+36\lambda_2^2)\pi^8}
	{1492992\lambda_1^8}.\nonumber
\end{align}
\end{proposition}
See \cite{Hwang2020} for an alternative approach to \eqref{E:an-ae}, 
based on Zagier's approach (which in turn relies on other identities 
and quantum modular forms).

We list the expressions of $d_1$ and $d_2$ in the two standard cases 
of FMs:

\medskip

\begin{center}
\begin{tabular}{c|c|c}
\multicolumn{1}{c}{} & 
\multicolumn{1}{c}{$\Lambda(z)=(1-z)^{-1}$}    
& \multicolumn{1}{c}{$\Lambda(z)=1+z$}  \\ \hline
$d_1$ & $\frac38-\frac{17\pi^2}{144}+\frac{\pi^4}{432}$
      & $\frac38+\frac{19\pi^2}{144}+\frac{\pi^4}{432}$\\
$d_2$ & $-\frac{7}{128}+\frac{17\pi^2}{1152}-\frac{59\pi^4}{41472}
	-\frac{5\pi^6}{62208}-\frac{5\pi^8}{1492992}$
      & $-\frac{7}{128}-\frac{19\pi^2}{1152}-\frac{35\pi^4}{41472}
	     +\frac{7\pi^6}{62208}-\frac{5\pi^8}{1492992}$
\end{tabular}    
\end{center}

\medskip

In particular, the expression $d_1$ is consistent with the expression 
given in \cite[p.\ 955]{Zagier2001}. 

\subsection{Asymptotics of the moments}

With the expansion \eqref{E:an-ae} available, we are now ready to 
derive the asymptotics of the first two moments and prove 
Theorem~\ref{T:mu-var}.

By \eqref{E:M1}, we have, as $z\sim 0$,
\[
    M_1(z) = \frac{M_0(z)-\Lambda(z)}{\Lambda(z)-1}
    = \Lpa{\frac1{\lambda_1z}-\frac{\lambda_2}{\lambda_1^2}
    +O(|z|)}M_0(z)+O(1).
\]
Then 
\[
    \mathbb{E}(X_n) = 
	\frac{[z^n]M_1(z)}{a_n}
	= \frac{a_{n+1}}{\lambda_1 a_n}
	-\frac{\lambda_2}{\lambda_1^2}
	+O\Lpa{\frac{a_{n-1}}{a_n}},
\]
which, together with \eqref{E:an-ae}, gives
\[
    \mathbb{E}(X_n)
	= \frac{a_{n+1}}{\lambda_1a_n}
    -\frac{\lambda_2}{\lambda_1^2} + O\lpa{n^{-1}}
    =\mu n +\frac{9}{\pi^2}
	-\frac{\lambda_2}{\lambda_1^2} + O\lpa{n^{-1}}.
\]
This proves \eqref{E:mu}, the first part of Theorem~\ref{T:mu-var}. 
Note that with the weaker form \eqref{E:an-counts}, the constant term 
cannot be made explicit. Further terms can be readily computed by 
computer algebra software; for example, using the expression of 
$d_1$ in \eqref{E:d1},
\[
	\mathbb{E}(X_n) = \mu \Lpa{n+\frac32} 
	-\frac{\lambda_2}{\lambda_1^2}
	+\frac1{n}\Lpa{\frac{1}{2\pi^2}-\frac{19}{24}
	-\frac{\pi^2}{72\lambda_1^2}
	+\frac{3\lambda_2}{2\lambda_1^2}
	-\frac{2\pi^2\lambda_3}{3\lambda_1^3}
	+\frac{2\pi^2\lambda_2^2}{3\lambda_1^4}}+O\lpa{n^{-2}}.
\]

Similarly, by \eqref{E:M2},
\[
	M_2(z) = \Lpa{\frac1{\lambda_1^2z^2}-
	\frac{\lambda_1^2+4\lambda_2}{2\lambda_1^3z}
	-\frac{\lambda_1^4-2\lambda_1^2\lambda_2+8\lambda_1\lambda_3
	-12\lambda_2^2}{4\lambda_1^4}+O(|z|)}M_0(z),
\]
and accordingly 
\begin{align*}
	\mathbb{V}(X_n)
	&= \frac{[z^n]M_2(z)}{a_n}-\mu_n^2\\
	&= \frac{a_{n+2}}{\lambda_1^2a_n}-
	\frac{(\lambda_1^2+4\lambda_2)a_{n+1}}{2\lambda_1^3a_n}
	-\frac{\lambda_1^4-2\lambda_1^2\lambda_2+8\lambda_1\lambda_3
	-12\lambda_2^2}{4\lambda_1^4}+O\Lpa{\frac{a_{n-1}}{a_n}}-\mu_n^2.
\end{align*}
By the expansion \eqref{E:an-ae} and \eqref{E:mu}, we then obtain 
\eqref{E:var} by straightforward calculations. A longer expansion is 
also easily computed; for example,
\[
    \mathbb{V}(X_n)
    = \sigma^2\Lpa{n+\frac32}-\frac14
	+\frac{\lambda_2}{2\lambda_1^2} 
	+\frac1n\left(\begin{array}{l}
		\frac{\pi^2}{48}-\frac1{4\pi^2}+\frac{19}{48}
		+\frac{\pi^2}{144\lambda_1^2}\\
		-\frac{3\lambda_2}{4\lambda_1^2}
		+\frac{\pi^2+12}{6}\lpa{\frac{\lambda_3}{\lambda_1^3}
		-\frac{\lambda_2^2}{\lambda_1^4}}
	\end{array}
	\right)+O\lpa{n^{-2}}.
\]

\section{Dimension of random FMs}
\label{S:euler} 

This section is devoted to a proof of Theorem~\ref{T:dimdis}, the
central limit theorem for the dimension of random $\Lambda$-FMs of
large size.

\subsection{A better bivariate generating function}

We begin with seeking a series representation of $F(z,v)$ better than
\eqref{E:Fzv0} because \eqref{E:Fzv0} contains negative coefficients
in the expansion of $1-\Lambda(z)^{-j}$. In addition to 
\eqref{E:Fzv0}, it is also known that (see \cite{Fu2020, Jelinek2012})
\begin{equation}\label{E:Fzv}
    \begin{split}
        F(z,v) 
        &=1+\sum_{k\ge1}\frac{v\Lambda(z)^{-k}}
        {1-v\lpa{1-\Lambda(z)^{-k}}}
        \prod_{1\le j\le k}\lpa{1-\Lambda(z)^{-j}},
    \end{split}
\end{equation}
but again the same sign problem occurs. A better expression for our 
purposes is the following one. 

\begin{lemma} The bivariate generating function for the dimension of 
$\Lambda$-FMs satisfies 
\begin{align}\label{E:Fzv4}
    F(z,v) = \sum_{n\ge0}P_n(v)z^n
    =1-v+\sum_{k\ge1}v^{k}\Lambda(z)^{k}
    \prod_{1\le j<k}\frac{\lpa{\Lambda(z)^j-1}^2}
    {1-(v-1)(\Lambda(z)^{j}-1)},
\end{align}
where $P_n(v)$ is the generating polynomial of dimension of
$\Lambda$-FMs of size $n$.
\end{lemma}
\begin{proof}
Substitute $t=1$, $s=\frac{v-1}{v(1-z)}$, $x=y=1-\Lambda(z)$ in the 
following identity of Andrews and Jel\'inek in \cite{Andrews2014}: 
\begin{equation}\label{E:aj}
    \begin{split}
        &\sum_{n\ge0}\frac{\lpa{\frac{s}{t(1-x)}; 1-x}_n  
        \lpa{\frac1{1-y};\frac1{1-x}}_n}{(s;1-x)_n} \,t^n\\
        &\qquad =(1-y)\sum_{n\ge0}
        \frac{(1-y;1-x)_n(t(1-x);1-x)_n}{(s;1-x)_n}\,(1-x)^n ,
    \end{split}
\end{equation}
where $(a;z)_n:=(1-a)(1-az)\cdots(1-az^{n-1})$. The left-hand side 
gives \eqref{E:Fzv}, and the right-hand side \eqref{E:Fzv4}.  
\end{proof}

Interestingly, this lemma gives a combinatorial interpretation of a
special case of the generalized Rogers-Fine identity \eqref{E:aj},
partially answering a question raised by Andrews and Jel\'inek
\cite{Andrews2014}.

\subsection{The exponential prototype}

As indicated above, we focus first on the special case when
$\Lambda(z) = e^z$, the general case being then deduced by an
argument based on change of variables.

Let
\begin{align}\label{E:exF1}
	E(z,v) = 1-v+\sum_{k\ge0}E_k(z)R_k(z,v),
\end{align}
where $E_k(z)$ is defined in \eqref{E:Ekz} and 
\begin{align}\label{E:Bkv}
	R_k(z,v) := v^{k+1}\prod_{1\le j\le k}
	\frac1{1-(v-1)\lpa{e^{jz}-1}}
    = v\prod_{1\le j\le k}
	\frac1{1-(1-v^{-1})e^{jz}}.
\end{align}
The first few terms in the Taylor expansion of $E(z,v)$ are 
\[
	E(z,v) = 1+vz + (v+2v^2)\frac{z^2}{2!}
	+(v+12v^2+6v^3)\frac{z^3}{3!}+
	(v+50v^2+132v^3+24v^4)\frac{z^4}{4!}+\cdots.
\]
Here the coefficient of $\frac{z^nv^k}{n!}$ counts \emph{the number 
of labelled \textbf{(2$+$2)} free posets of $n$ elements and with 
magnitude $k-1$}.

While the Taylor expansion of $R_k(z,v)$ (in $z$ and $v$) still
contains, in general, negative coefficients, the series \eqref{E:Bkv}
is suitable for our purposes because $R_k$ plays asymptotically only a
perturbative role when $v$ is close to $1$ in view of the estimate
\begin{align}\label{E:Bk-ub}
	R_k(z,v) =\prod_{1\le j\le k}
	\frac1{1+O(|v-1|)}
	= e^{O(k|v-1|)},
\end{align}
for $1\le k= O(n)$ and small $z\asymp n^{-1}$ when $\lambda_1>0$, 
while, in the same setting, 
\[
	E_k(z) 
	= e^{(k+1)z}\prod_{1\le j\le k}\lpa{e^{jz}-1}^2
    = e^{\Omega(k)};
\]
see below for more precise analysis.

Our aim is to prove the asymptotic normality of the random 
variable $X_n$, which, in the case of \eqref{E:exF1}, is defined as 
\begin{align}\label{E:Xn-k}
	\mathbb{P}(X_n=k)
	:= \frac{n![z^nv^k]E(z,v)}{n![z^n]E(z,1)}
	= \frac{[z^nv^k]E(z,v)}{[z^n]E(z,1)}
	\qquad(1\le k\le n),
\end{align}
for $n\ge1$, and $X_n$ assumes only integer values. For that purpose,
we will restrict our analysis to the range $|v-1|\le\ve$,
$v\in\mathbb{C}$. Then, according to the approach sketched in 
\S~\ref{S:2stage} for the asymptotics of $e_n$, we would expect, when 
$|v-1|\le\ve$, that
\begin{align*}
	[z^n]E(z,v) 
	&= [z^n]\sum_{0\le k\le\tr{\frac12n}} E_k(z)R_k(z,v)
	\approx \sum_{0\le k\le\tr{\frac12n}} R_k(r,v) [z^n]E_k(z)\\
	&\approx R_{\tr{qn}}\lpa{(\mu n)^{-1},v} [z^n]E(z,1),
\end{align*}
where $q=\mu\Log 2$. This is, up to the leading constant, correct 
because
\[
	R_{\tr{qn}}\lpa{(\mu n)^{-1},e^{i\theta/\sqrt{n}}}
	\sim e^{\mu \sqrt{n}\,i\theta - \frac3{2\pi^2}\,\theta^2},
\]
and we see that while the coefficient of $i\theta$ matches that of
the mean, the coefficient of $\theta^2$ is not equal to
$-\frac12\sigma^2$, as desired, showing that a more delicate analysis 
is required.

\begin{proposition} \label{P:Xn-cf}
For large $n$, the coefficient of $z^n$ in the Taylor expansion of 
$E(z,v)$ defined in \eqref{E:exF1} satisfies 
\begin{align}\label{E:Pnv-ue}
	[z^n]E\lpa{z,e^{i\theta/\sqrt{n}}}
	= cn^{\frac12}\mu^n n! \Exp\Lpa{\mu\sqrt{n}\, i\theta
	-\frac{\sigma^2\theta^2}2}
	\lpa{1+O\lpa{(|\theta|+|\theta|^3)n^{-\frac12}}},
\end{align}	
uniformly for $\theta=o(n^{\frac16})$, where 
$c=\frac{12\sqrt{3}}{\pi^{5/2}}$ and $(\mu,\sigma^2)$ is defined 
in \eqref{E:dimdis}.
\end{proposition}

The approximation \eqref{E:Pnv-ue} implies, by \eqref{E:Xn-k} 
($\mu_n := \mathbb{E}(X_n)$ and $\sigma_n^2 := \mathbb{V}(X_n)$),
\begin{align}\label{E:mgf-cv}
	\mathbb{E}\llpa{\Exp\Lpa{\frac{X_n-\mu_n}{\sigma_n}\,i\theta}}
	\to e^{-\frac12\theta^2},
\end{align}
uniformly for $\theta=O(1)$, and then the asymptotic normality of
$X_n$ \eqref{E:dimdis} (when $\Lambda(z)=e^z$) follows from standard
convergence theorem for characteristic functions; see, e.g.,
\cite[p.\ 777]{Flajolet2009}. 

Throughout this section, $v\in\mathbb{C}$ always lies in a small 
neighborhood of unity, $|v-1|\le\ve$, unless otherwise indicated. 

\subsection{The factorial growth order}

The approach we adopt here follows mostly that sketched above in
Section~\ref{S:spm} from \cite{Hwang2020} but is carried out
differently, with a particular attempt to keep it more self-contained.
We begin with the following lemma, showing that a simple inequality is
already sufficient to characterize the factorial growth of the
problem; furthermore, it shows that the sum of $[z^n]E_k(z) R_k(z,v)$
over the ranges $k\le \alpha_-n$ and $k\ge \alpha_+n$ is
asymptotically negligible, where $\alpha_\pm$ are specified below.

\begin{proposition} Let $\alpha_\pm>0$ be the two zeros of the 
equation $\varphi(\alpha)=\Log\mu$, where 
\begin{align}\label{E:varphi}
	\varphi(\alpha) := 2(1-\alpha)\Log\alpha
	-(1-2\alpha)\Log(1-2\alpha)+2-4\alpha.	
\end{align}
Then 
\begin{align}\label{E:sum-a-pm}
	\llpa{\sum_{0\le k\le (\alpha_--\ve)n}
	+\sum_{k\ge(\alpha_++\ve)n}} [z^n]E_k(z)R_k(z,v)
	= O\lpa{n!(\mu-\ve)^n},
\end{align}	
uniformly for $|v-1|=o(1)$. 
\end{proposition}
\begin{proof}
Since the Taylor expansion of $E_k(z)$ contains only positive 
coefficients, we have, by the elementary inequality $e^x-1\le xe^x$ 
for $x\ge0$,
\begin{align*}
	e_{n,k} &\le r^{-n}E_k(r)
	=r^{-n}e^{(k+1)r}\prod_{1\le j\le k}\lpa{jre^{jr}}^2
    =r^{-n}e^{(k+1)^2r}k!^2r^{2k}\qquad(r>0).
\end{align*}
The optimal choice of $r>0$ at which the right-hand side reaches its
minimum value for fixed $n$ and $k$ is obtained by taking derivative
with respect to $r$, setting it equal to zero and then solving for 
$r$. In this way, we find that the minimum such $r$ is $r = r_0 = 
(n-2k)(k+1)^{-2}$, and we get, with $k=\alpha n$, $0<\alpha<\frac12$,
\begin{align}\label{E:enk-ub1}
	e_{n,k}
	\le r_0^{-n}E_k(r_0)
	=O\lpa{k n^n e^{(\varphi(\alpha)-1)n}},
\end{align}
by Stirling's formula, where $\varphi$ is defined in 
\eqref{E:varphi}. When $\alpha=\frac12$ or $n-2k=\ell=o(n)$, we take 
$r=(1+\ell)n^{-2}$, giving the estimate 
$e_{n,k}=O\lpa{n^{n+2\ell}(2e)^{-n}n^{2\ell}(\ell+1))^{-2\ell}}$. 
Then 

\noindent
\begin{minipage}{0.73\textwidth}
\[
	e^{\varphi(\alpha)} < \mu
	\quad\text{if and only if} \quad
	0\le\alpha<\alpha_- \;\text{and}\; \alpha_+<\alpha\le\frac12,
	\qquad\quad
\]
where $\alpha_\pm>0$ solves the equation $\varphi(\alpha) = \Log\mu$.
Numerically, $\alpha_-\approx 0.30686$ and $\alpha_+\approx 
0.46628$. On the other hand, when $|v-1|=o(1)$, we have, by 
\eqref{E:Bk-ub} and \eqref{E:enk-ub1},
\end{minipage}\;\;
\begin{minipage}{0.25\textwidth}\centering
	\includegraphics[width=0.9\textwidth]{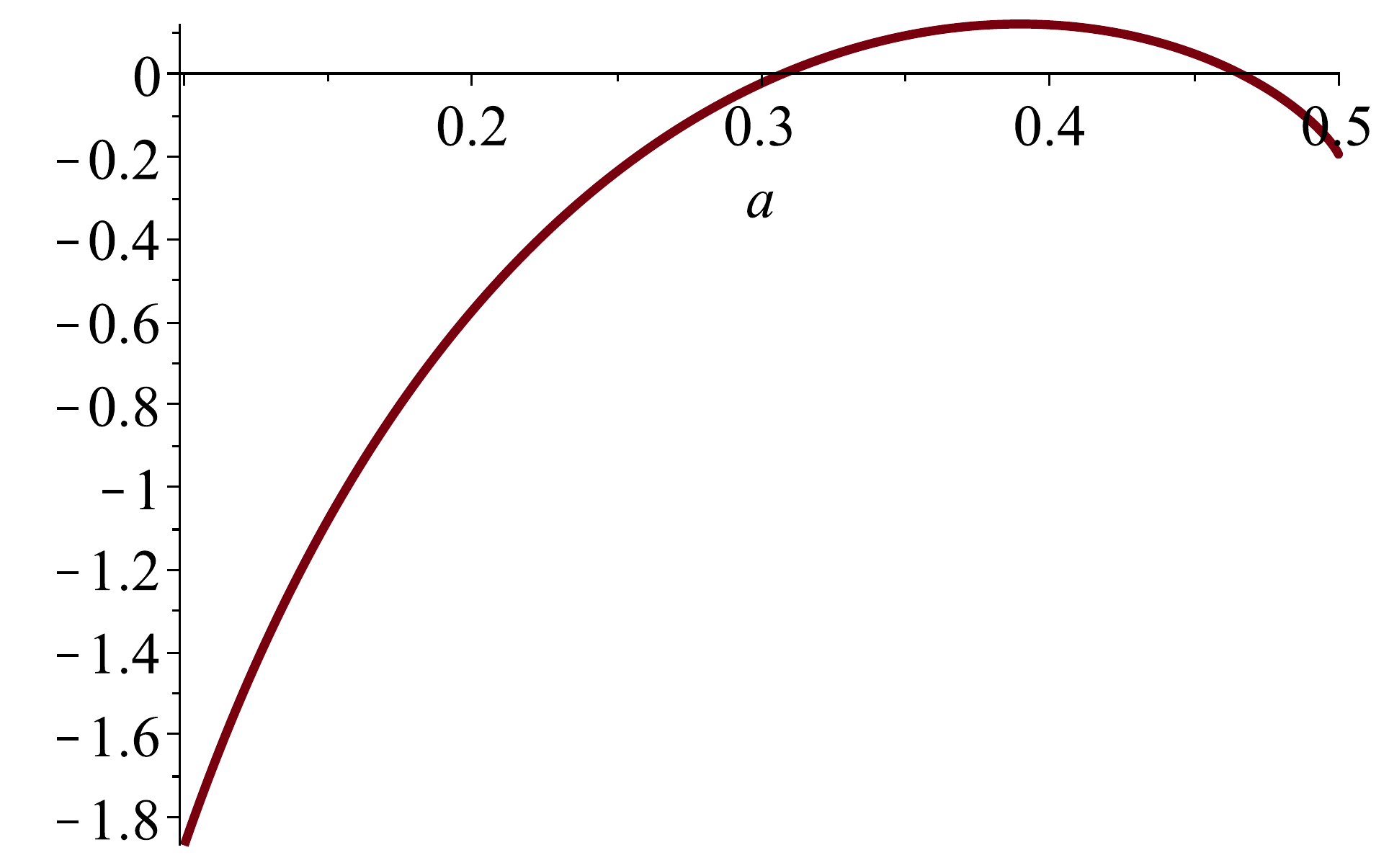}
	
	\centering{\emph{A plot of $\varphi(\alpha)-\Log\mu$.}}
\end{minipage}


\begin{align*}
	[z^n]E_k(z) R_k(z,v)
	= O\lpa{\Max\limits_{|z|=r}|R_k(z,v)|r^{-n}E_k(r)}
	= O\lpa{k n^n e^{(\varphi(\alpha)-1+o(1)) n}}.
\end{align*}
It follows that
\[
	\llpa{\sum_{0\le k\le (\alpha_--\ve)n}
	+\sum_{k\ge(\alpha_++\ve)n}} [z^n]E_k(z) R_k(z,v)
	= O\lpa{n! n^{\frac32}\lpa{e^{\varphi(\alpha_--\ve)n} 
	+e^{\varphi(\alpha_++\ve)n)}}}.
\]
This proves \eqref{E:sum-a-pm} since $\varphi(\alpha_\pm)=\Log\mu$ 
and $\varphi(\alpha)$ is concave on $[0,\frac12]$.
\end{proof}

With the estimate \eqref{E:sum-a-pm} available, we will limit our 
asymptotic study of $[z^n]E_k(z) R_k(z,v)$ to only linear $k$, namely 
$\ve n\le k\le (\frac12-\ve)n$. 

\subsection{The exponential growth order}

In this section, we derive the exponential term in the growth rate of 
$[z^n]E(z,v)$, starting from another simple (finer) approximation of 
$\Log E_k(r)$. We also establish the asymptotic negligibility of  
$[z^n]E_k(z)R_k(z,v)$ for $k$ outside $[k_-,k_+]$, where $k_\pm$ is 
defined below in Proposition ~\ref{P:an2}. 

\begin{lemma} For $0\le k\le \tr{\frac12n}$,
\begin{equation}\label{E:Ekr-2}
\left\{
\begin{split}
	\Log E_k(r) 
	&\le 2k\Log\lpa{e^{kr}-1} -2\Log(e^r-1)
	- \int_0^k \frac{2xr}{1-e^{-xr}}\,\dd x + O(1+kr),\\
	\Log R_k\lpa{z,e^{i\theta}}
	&= \frac{e^{kz}-1}{z}(1-e^{-i\theta})
	+\frac{e^{2kz}-1}{4z}(1-e^{-i\theta})^2
	+O\lpa{|z|^{-1}|\theta|^3+|\theta|+|z|},
\end{split}	\right.
\end{equation}
uniformly for $|z|=r=O(n^{-1})$, and $\theta=o(1)$. 
\end{lemma}
\begin{proof}
By the monotonicity of $\Log(e^x-1)$, 	
\begin{align*}
	\Log E_k(r) 
	= (k+1)r+2\sum_{1\le j\le k}\Log\lpa{e^{jr}-1}
	\le 2\int_1^k\Log\lpa{e^{xr}-1}\dd x +O(1+kr),
\end{align*}
which then proves the upper bound of $\Log E_k(r)$ in \eqref{E:Ekr-2} 
by an integration by parts. Note that the integral in \eqref{E:Ekr-2} 
equals $2r^{-1}I(kr)$; see \eqref{E:Iz}. 

The other approximation in \eqref{E:Ekr-2} is obtained by the 
expansion
\begin{align*}
	R_k\lpa{z,e^{i\theta}}
	&= e^{i\theta}
	\prod_{1\le j\le k}\frac1{1-(1-e^{-i\theta}) e^{jz}}
	= \Exp\llpa{i\theta+\sum_{l\ge1}
	\frac{(1-e^{-i\theta})^l}{l}
	\cdot \frac{e^{klz}-1}{1-e^{-lz}}}. \qedhere
\end{align*}
\end{proof}

\begin{proposition} \label{P:an2}
With $q=\mu\Log 2\approx 0.42138$ and $k_\pm := qn \pm 2\,\varsigma 
n^{\frac58}$, where
\[
	\varsigma^2 := \frac{3(24(\Log 2)^2-\pi^2)}{2\pi^4}, 
\] 
we have 
\begin{align}\label{E:an-3}
	\llpa{\sum_{0\le k\le k_-}
	+\sum_{k_+\le k\le 0.5n}}[z^n]E_k(z) R_k(z,e^{i\theta})
	= O\lpa{n! \mu^ne^{-n} n^2 e^{-\ve n\theta^2 - n^{\frac14}}},	
\end{align}
uniformly for $\theta=o(1)$. 	
\end{proposition}
\begin{proof}
By \eqref{E:Ekr-2}, we take $r$ to be the positive solution of the
equation $2I(kr)=rn$, which exists as long as $0\le k<\tr{\frac12n}$
by the asymptotic behaviors of $I(x)$ for small and large $x$, and
satisfies $r\asymp (n-2k)k^{-2}$; see \cite{Hwang2020}. Then, with
$k=\alpha n$ and $r=\xi n^{-1}$, we have
\[
	r^{-n}E_k(r) 
	= O\lpa{n^2e^{-n\Log r+2k\Log(e^{kr}-1)-n}}
	= O\lpa{n^{n+2} e^{\phi(\alpha,\xi)n}},
\] 
\noindent
\begin{minipage}{0.73\textwidth}
where $\phi(\alpha,\xi)$ is defined in (\ref{E:enk-spb}) subject to
the condition $2I(\alpha\xi)=\xi$. The maximum value of
$\phi(\alpha,\xi)$ is characterized by computing the solution of the
equation $\partial_\alpha \phi(\alpha,\xi)=0$, which is reached at
$(\alpha,\xi)= (\mu\Log 2,\mu^{-1})$. Then we obtain $\phi(\mu\Log
2,\mu^{-1}) = \Log\mu-1$. In addition, it is easy to prove the
concavity of $\phi(\alpha,\xi)$ when viewed as a function of
$\alpha$; see \cite[Lemma 11]{Hwang2020}.
\end{minipage}\;\;
\begin{minipage}{0.25\textwidth}\centering
	\includegraphics[width=0.9\textwidth]{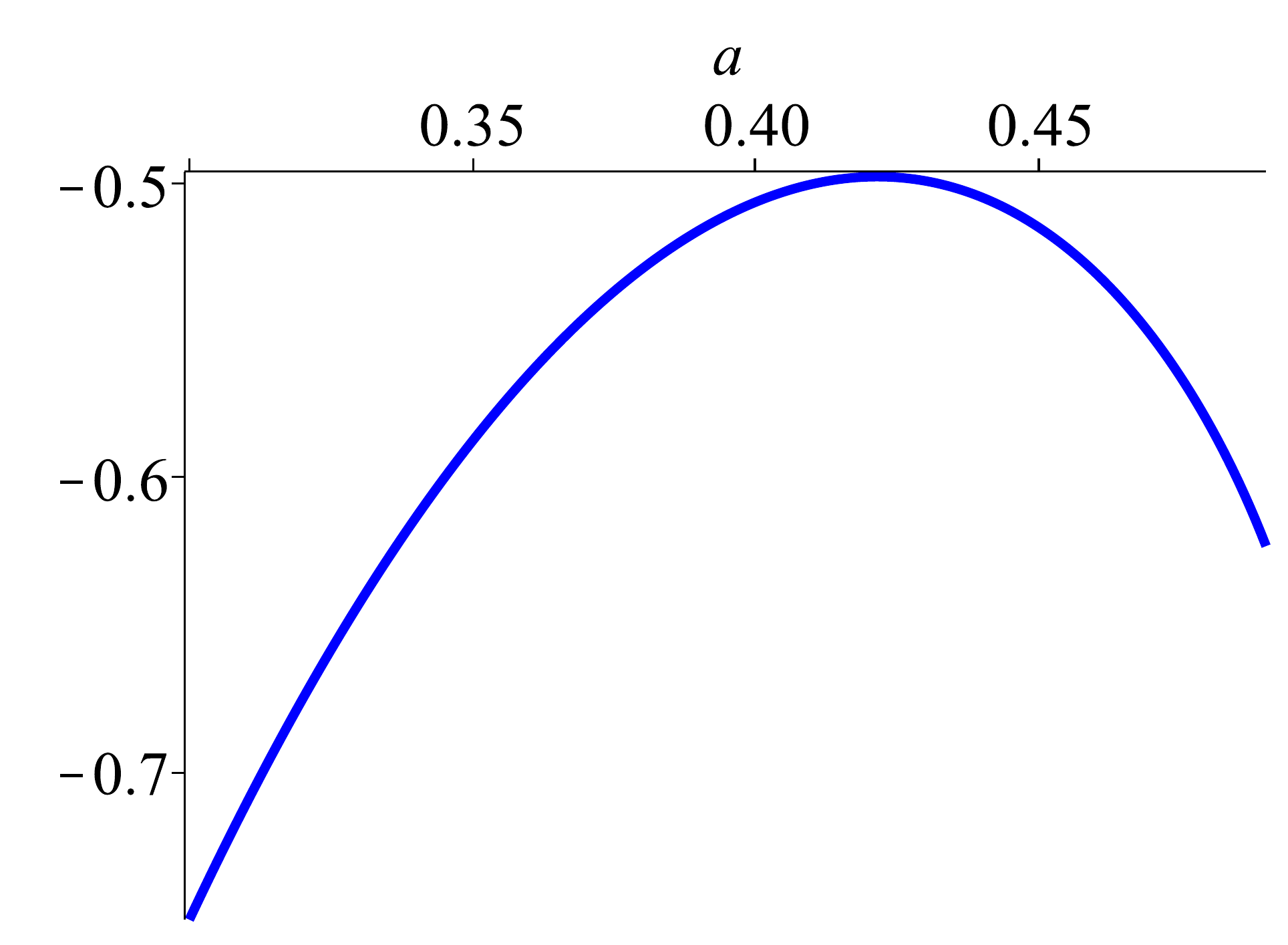}
	
	\centering{\emph{A plot of $\phi(\alpha,\xi)+1$; $\Log\mu\approx 
	-0.4977$.}}
\end{minipage}

\medskip

Now with the estimates in \eqref{E:Ekr-2} and the 
inequality \eqref{E:Ekz-o}, we obtain 
\begin{align}\label{E:snkv-1}
	[z^n]E_k(z) R_k(z,v)
	= O\llpa{r^{-n}E_k(r)
	\int_{-\pi}^{\pi} |R_k(re^{it},v)|
	\Exp\Lpa{-\frac{(k+1)^2rt^2}{\pi^2}}
	\dd t},
\end{align}
where $r>0$ is chosen to be the same as above, namely, $2I(kr)=rn$. 
Since
\begin{align*}
	\Log |R_k(z,e^{i\theta})| 
	= O\lpa{|\theta|r^{-1}}
	= o((k+1)^2 r),
\end{align*}
when $\theta=o(1)$, $k\asymp n$ and $r\asymp n^{-1}$, and the 
integral has the typical form amenable to the Laplace's method, we then 
deduce, by using the local expansion
\begin{align}\label{E:Bk-le}
	\Log\bigl|R_k\lpa{re^{it},e^{i\theta}}\bigr|
	= -\frac{(e^{kr}-1)^2}{4r}\,\theta^2
	-r\partial_r\Lpa{\frac{e^{kr}-1}{r}}\,\theta t
	+O\lpa{r^{-1}(\theta^4+\theta^2t^2+|t||\theta|^3)},
\end{align}
that 
\begin{align}\label{E:int-BE}
	\int_{-\pi}^{\pi} |R_k(re^{it},v)|
	\Exp\Lpa{-\frac{(k+1)^2rt^2}{\pi^2}}\dd t
	&= O\llpa{\frac{e^{-c'(1+o(1))n\theta^2}}{(k+1)\sqrt{r}}},
\end{align}
where
\begin{align*}
	c' := \frac{(e^{kr}-1)^2}{r}\,
	-\frac{\pi^2(e^{2kr}(kr-1)^2+2e^{kr}(kr-1)+1)}
	{4(k+1)^2r^3}.
\end{align*}
Now, for $k\sim qn$ and $r\sim \mu n^{-1}$, we get
\[
	c' \sim \frac{-3(\Log 2)^2+4\Log 2-1}{23(\Log 2)^2}
	\approx 0.2835.
\]
We next improve on the growth order of $r^{-k}E_k(r)$ when $k\sim 
qn$. Write 
\begin{align}\label{E:k-qn}
	k=qn + \varsigma \sqrt{n}x
	\qquad(x=o(\sqrt{n})).
\end{align}
Solving the equation $2I(q\xi)=\xi$ for $\xi$ gives the expansion 
$r = \xi n^{-1}$, where
\[
	\xi = \frac1{\mu}-\frac{\Log2}{\sigma\sqrt{n}}\, x
	+\frac{4\pi^2\sigma^2(3+2\pi^2\sigma^2)-9(1-\Log 2)}
	{48\pi^2\sigma^4 n}\,x^2+O\lpa{|x|^3n^{-\frac32}}.
\]
With these expansions, we then obtain
\[
	\phi(\alpha,\xi)
	= \Log\mu -\frac{x^2}{2}+\frac{\pi^4-72\pi^2\Log 2
	+1152(\Log 2)^3}{16\pi^6\sigma^3\sqrt{n}}\, x^3
	+O\lpa{x^4n^{-1}}.
\]
This implies that as long as $k$ satisfies \eqref{E:k-qn}, we have
\[
	r^{-n}E_k(r) 
	= O\lpa{n^{n+2}\mu^n e^{-n-\frac12 x^2+O(|x|^3n^{-\frac12}})},
\]
uniformly for $x=o(\sqrt{n})$. Combining this with the estimates 
\eqref{E:snkv-1} and \eqref{E:int-BE}, we then have
\[
	|[z^n]E_k(z) R_k(z,e^{i\theta})| 
	= O\lpa{n^{n+\frac32}\mu^n
	e^{-n-\ve n\theta^2-\frac12x^2(1+o(1))}},
\]
uniformly for $x=o(\sqrt{n})$. Thus, by the monotonicity of 
$\alpha\mapsto \phi(\alpha,\xi)$, we obtain \eqref{E:an-3}.
\end{proof}

\subsection{The asymptotic equivalent and the proof of Proposition~\ref{P:Xn-cf}}

We complete the proof of Proposition~\ref{P:Xn-cf} in this section.
The analysis is similar to that conducted in the proof of
Proposition~\ref{P:an2}, and will be brief.

We first derive a more precise approximation to $\Log E_k(z)R_k(z,v)$
than \eqref{E:Ekr-2} for $v$ close to $1$.
\begin{lemma} \label{C:log-Bk}
For $z\in\mathbb{C}$, $z\ne0$, we have 
\begin{equation}\label{E:logaka}
\begin{split}
	\Log E_k(z)R_k\lpa{z,e^{i\theta}}
	&=2k\Log(e^{kz}-1)-\frac{2I(kz)}z
	+\frac{(e^{kz}-1)i\theta}{z}
    +\Log\frac{2\pi(e^{kz}-1)}{z}\\ 
	&\quad +kz
    -\frac{(e^{kz}-1)^2\theta^2}{4z}
    +O\lpa{k^{-1}+|\theta|+n|\theta|^3},
\end{split}
\end{equation}
uniformly for $\theta=o(1)$, $|z|\asymp n^{-1}$ and 
$1\le k\le \tr{\frac12n}$.
\end{lemma}
\begin{proof}
The expansion \eqref{E:logaka} follows from the Euler-Maclaurin 
formula. 

\end{proof}

We also need the following estimate for the tail of a Gaussian 
integral. 
\begin{lemma} Assume $K,t_1>0$ and $L\in\mathbb{R}$. Then  
\begin{align}\label{E:int-KL}
	\int_{t_1}^\infty e^{-Lt-Kt^2}\dd t
	= O\lpa{e^{-t_1(Kt_1+L)}(2Kt_1+L)^{-1}},
\end{align}
as long as $(2Kt_1+L)K^{-\frac12}\to\infty$.
\end{lemma}
\begin{proof}
By a direct change of variables $u=Lt+Kt^2+\frac{L^2}{4K}$,
\[
	\int_{t_1}^\infty e^{-Lt-Kt^2}\dd t
	= \frac{e^{L^2/(2K)}}{2\sqrt{K}}
	\int_{t_2}^\infty e^{-u}u^{-\frac12}\dd u,
\]	
where $t_2 := (2Kt_1+L)^2(4K)^{-1}$. Now we have, as $t_2\to\infty$,
\[
	\int_{t_2}^\infty e^{-u}u^{-\frac12}\dd u
	\sim t_2^{-\frac12}e^{-t_2}. \qedhere
\]
\end{proof}
Now we are ready to complete the proof of Proposition \ref{P:Xn-cf}.

In what follows, write $v=e^{i\theta/\sqrt{n}}$. Then
\[
	[z^n]E_k(z)R_k\lpa{z,e^{i\theta/\sqrt{n}}}
	= \frac{r^{-n}}{2\pi}\llpa{\int_{|t|\le t_0}
	+\int_{t_0<|t|\le\pi}} e^{-int}E_k(re^{it}) 
	R_k\lpa{re^{it},e^{i\theta/\sqrt{n}}}\dd t,
\]
where $t_0:=6n^{-\frac38}$. Since $t_0$ is small, the second integral 
is estimated by the same arguments used above for \eqref{E:snkv-1}, 
and we are led to an integral of the form \eqref{E:int-KL} with 
\[
	K=\frac{(k+1)^2r}{\pi^2}, \quad
	L=\frac{r\theta}{\sqrt{n}}\partial_r\Lpa{\frac{e^{kr}-1}{r}},
	\AND t_1 = t_0.
\]
Since  
\[
	\frac{Kt_1+L}{\sqrt{K}}
	\sim \frac{6\sqrt{6}\Log 2}{\pi^2}\, n^{\frac18}
	>n^{\frac18},
\]
when $k\sim qn$ and $r\sim (\mu n)^{-1}$, we have, by 
\eqref{E:snkv-1}, \eqref{E:Bk-le} and \eqref{E:int-KL}, 
\begin{align*}
	&r^{-n}\int_{t_0<|t|\le\pi} e^{-int}E_k(re^{it}) 
	R_k\lpa{re^{it},e^{i\theta/\sqrt{n}}}\dd t\\
	&\qquad= O\llpa{r^{-n}E_k(r) 
	\int_{t_0}^\pi |R_k\lpa{re^{it},e^{i\theta/\sqrt{n}}}|
	\Exp\Lpa{-\frac{(k+1)^2rt^2}{\pi^2}}\dd t}\\
	&\qquad= O\lpa{r^{-n}E_k(r)(k+1)^{-1} r^{-\frac12}
	e^{-\ve \theta^2-n^{\frac14}}},
\end{align*}
which will be seen to be negligible. 

Denote by 
\[
	\Xi_2(r) := 2[s^2]
	\Log \lpa{E_k(re^s) R_k\lpa{re^s,e^{i\theta/\sqrt{n}}}},
\]
where $r>0$ solves the saddle-point equation $[s]\Log 
E_k(re^s)R_k\lpa{re^s,e^{i\theta/\sqrt{n}}} =n$. The analysis we 
carried out so far implies that the saddle-point approximation 
\begin{align}\label{E:spa}
	[z^n]E_k(z) R_k\lpa{z,e^{i\theta/\sqrt{n}}}
	= \frac{r^{-n}E_k(r)R_k\lpa{r,e^{i\theta/\sqrt{n}}}}
	{\sqrt{2\pi \Xi_2(r)}}
	\Lpa{1+O\lpa{\Xi_2(r)^{-1}}},
\end{align}
is well-justified for $k_-\le k\le k_+$, provided that 
$\Xi_2(r)\to\infty$, which will be seen to be the case. Here $k_\pm$ 
is defined in Proposition~\ref{P:an2}.

By the Euler-Maclaurin formula or simply \eqref{E:logaka}, the 
saddle-point equation is given asymptotically by 
\begin{align*}
	&\frac{I(kr)}r-n+
	\frac{e^{kr}(kr-1)+1}{r\sqrt{n}}\,i\theta\\
	&\qquad +\frac1{rn}\llpa{\frac{kr^2n}{e^{kr}-1}+rn(2kr-1)
	-\frac{(e^{kr}-1)(e^{kr}(2kr-1)+1)\theta^2}
	{4}}+O\lpa{n^{-1}}=0.
\end{align*}
The remaining steps are readily coded, and we sketch only the main 
steps. 
 
Write $k=qn+\varsigma \sqrt{n}\,x$. Assuming an expansion of the 
form 
\begin{align}\label{E:r}
    r=\sum_{j\ge0}r_j n^{-\frac12j-1},
\end{align}
where $r_0 = \mu^{-1}$, we then derive, by a standard bootstrapping 
argument, that
\[
	r_1 = -\frac{\Log 2}{\varsigma}\, x
	-\frac{3(2\Log 2-1)}{2\pi^2\varsigma^2}\,i\theta,
\]
and $r_j$ is a polynomial of $x$ and $\theta$ of degree $j$ (the 
expressions are very messy for $j\ge2$). Substituting this 
expansion and $k=qn+\varsigma \sqrt{n}\,x$ into \eqref{E:spa} using 
the asymptotic approximation \eqref{E:logaka}, we then deduce that 
\begin{align*}
	\frac{r^{-n}E_k(r)R_k\lpa{r,e^{i\theta/\sqrt{n}}}}
	{c_0 \mu^n e^{-n}n^{n+1} 
	e^{\mu\sqrt{n}\,i\theta-\frac12\sigma^2\theta^2}}
	= e^{-\frac12\lpa{x+\frac{3(\pi^2-12\Log 2)i\theta}
	{\pi^4\varsigma}}^2}
	\Lpa{1+O\Lpa{\frac{|x|+|\theta|
	+(|x|+|\theta|)^3}{\sqrt{n}}}},
\end{align*}
uniformly for $x,\theta=o(n^{\frac16})$, where $c_0=24\pi^{-1}$.

On the other hand, 
\[
	\Xi_2(r) = \frac{e^{kr}(2k^2r-n)+n}{e^{kr}-1}
	+\frac{k^2re^{kr}}{\sqrt{n}}\,i\theta
	+O(1),
\]
or, by substituting the expansions of $k$ and $r$, 
\[
	\Xi_2(r) = \frac{2\pi^2\varsigma^2}{3}\, n
    +O\lpa{(|x|+|\theta|)\sqrt{n}}.
\]
Collecting these expansions to (\ref{E:spa}) yields
\begin{align*}
	\frac{[z^n]E_k(z)R_k\lpa{z,e^{i\theta/\sqrt{n}}}}
	{c\mu^ne^{-n} n^{n+\frac12}
    e^{\mu\sqrt{n}\,i\theta-\frac12\sigma^2\theta^2}}
	= e^{-\frac12\lpa{x+\frac{3(\pi^2-12\Log 2)i\theta}
	{\pi^4\varsigma}}^2}
	\Lpa{1+O\Lpa{\frac{|x|+|\theta|
	+(|x|+|\theta|)^3}{\sqrt{n}}}},
\end{align*}
where $c=12\sqrt{3}\pi^{-\frac52}$, as in Proposition~\ref{P:Xn-cf}. 

Summing over $k_-\le k\le k_+$ and approximating the sum by a 
Gaussian integral, we then produce an extra factor of the form 
\[
	\sqrt{2\pi} 
	\,\varsigma\sqrt{n}
    \lpa{1+O\lpa{(|\theta|+|\theta|^3)n^{-\frac12}}}.
\]
By Stirling's formula for the factorial, we obtain 
\eqref{E:Pnv-ue}, which completes the proof of 
Proposition~\ref{P:Xn-cf}. 

\subsection{Proof of Theorem~\ref{T:dimdis}}

We now translate the asymptotic approximation \eqref{E:Pnv-ue} for 
$[z^n]E(z,v)$ when $\Lambda(z)=e^z$ into that for $[z^n]F(z,v)$ (see 
\eqref{E:Fzv}) in the more general situations with $\lambda_1>0$. By 
\eqref{E:Fzv4}, 
\begin{align*}
    P_n(v)
	:=[z^n]F(z,v)
	=[z^{n}]\sum_{0\le k\le \tr{\frac12n}}v\Lambda(z)
	\prod_{1\le j\le k}
	\frac{\Lambda(z)(\Lambda(z)^{j}-1)^2}
	{1-(1-v^{-1})\Lambda(z)^{j}}.
\end{align*}
Since $\Lambda(z)=1+\lambda_1z+\cdots$ with $\lambda_1\ne 0$ is 
analytic at $z=0$, the function is locally invertible at $z=0$ and we 
can make the change of variables $\Lambda(z)=e^y$, namely, let 
$z=\beta(y)$ so that $\Lambda(\beta(y))=e^y$ ($\beta(y)$ also 
analytic at $y=0$). Then we have
\[
    P_n(y)
    = [y^n]\Psi_n(y,v)\sum_{0\le k\le \tr{\frac12n}}E_k(y)R_k(y,v),
\]
where $\Psi_n(y) := \beta(y)^{-n-1} y^n \beta'(y)$. Now the expansion 
\[
    \beta(y) = \sum_{j\ge1}\beta_j y^j,\WITH
    \beta_1 := \frac{1}{\lambda_1}, \AND
    \beta_2 := \frac1{\lambda_1}\llpa{\frac12
    -\frac{\lambda_2}{\lambda_1^2}},
\]
implies that 
\begin{align}
    \beta'(y)\beta(y)^{-n-1}
    &=\lambda_1^ny^{-n-1} e^{-\beta_2 ny/\beta_1}
    \lpa{1+O(|y|+n|y|^2)}, \label{E:ztoal}
\end{align}
for small $|y|$. Consequently, when $y\asymp n^{-1}$ is small, we get 
\begin{align}\label{E:an}
    \Psi_n(y)&= \lambda_1^n [y^n]
    e^{-\beta_2 ny/\beta_1}\lpa{1+O(|y|+n|y|^2)}
    \sum_{0\le k\le \tr{\frac12n}}E_k(y)R_k(y,v).
\end{align}
Following the proof of Proposition~\ref{P:Xn-cf}, we know that only a 
small neighborhood of $(k,y) \sim \lpa{qn, (\mu n)^{-1}}$ 
contributes dominantly; furthermore, the extra factor before the sum 
in \eqref{E:an} is bounded when $y\asymp n^{-1}$. Thus we substitute 
$y=r$, with $r$ expanding as in \eqref{E:r}, and obtain 
\begin{align}\label{E:exp1}
    e^{-\beta_2 ny/\beta_1}
    &= e^{-\beta_2/(\beta_1\mu)}
    \lpa{1+O(n^{-\frac12})}.
\end{align}
This, together with \eqref{E:an} and \eqref{E:Pnv-ue} of 
Proposition~\ref{P:Xn-cf}, implies that
\begin{align*}
    P_n\lpa{e^{i\theta/\sqrt{n}}}
    &=\lambda_1^n e^{-\beta_2/(\beta_1\mu)}
    \lpa{1+O\lpa{n^{-\frac12}}}
    [y^n]E\lpa{y,e^{i\theta/\sqrt{n}}}\\
    &\qquad +O\Lpa{\bigl|\lpa{[y^{n-1}]+n[y^{n-2}]}
    E\lpa{y,e^{i\theta/\sqrt{n}}} \bigr|}\\
    &=c \sqrt{n}(\mu\lambda_1)^nn! e^{\mu\sqrt{n}i\theta
    -\frac12\sigma^2\theta^2}
    \lpa{1+O\lpa{(1+|\theta|+|\theta|^3)n^{-\frac12}}},
\end{align*}
where 
\[
    c := \frac{12\sqrt{3}}{\pi^{5/2}}
    \,e^{\frac{\pi^2}6
    \lpa{\frac{\lambda_2}{\lambda_1^2}-\frac12}}.
\]
Now
\[
    \mathbb{E}\Lpa{e^{(X_n-\mu n)i\theta/\sqrt{\sigma^2n}}}
    = \frac{P_n\lpa{e^{i\theta/\sqrt{n}}}}{P_n(1)}
    = e^{\mu\sqrt{n}i\theta
    -\frac12\sigma^2\theta^2}
    \lpa{1+O\lpa{(1+|\theta|+|\theta|^3)n^{-\frac12}}},
\]
implying Theorem~\ref{T:dimdis} by the continuity theorem for 
characteristic functions. 

\section{Proof of the Stoimenow conjecture (Theorem \ref{T:genzag})}\label{S:sto}

We begin by reviewing some necessary definitions on the chord
diagrams for our purposes. A \emph{regular linearized chord diagram
(regular LCD)}, also known as a Stoimenow matching, is a matching of
the set $[2n]=\{1,2,\ldots,2n\}$, namely, it is a partition of $[2n]$
into subsets of size exactly two. Each of the subsets is called an
\emph{arc}. A matching is a \emph{regular LCD} if it has \emph{no}
nested pairs of arcs such that either the openers or the closers are
next to each other.

The number of regular linearized chord diagram of size $n$ (length 
$2n$) equals the $n$-th Fishburn number $f_n$; see for instance 
\cite{Bousquet-Melou2010,Zagier2001} and Figure \ref{F:b3}.
\begin{figure}[ht]
	\centering
	\includegraphics[scale=0.6]{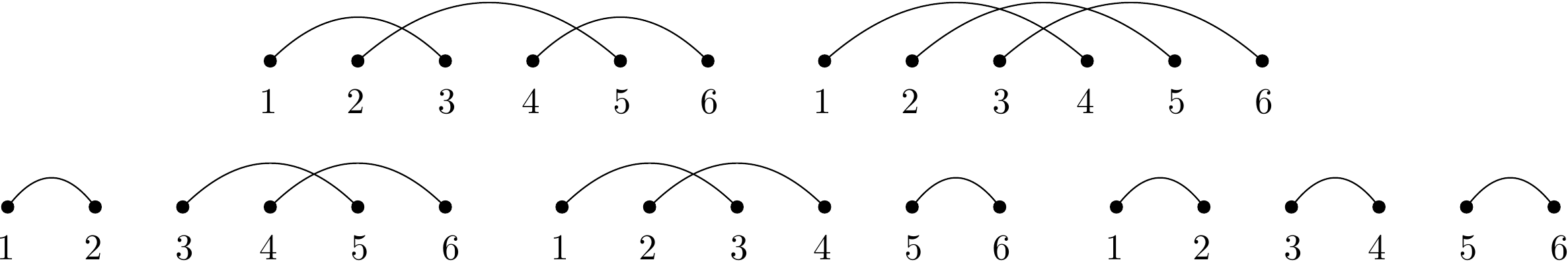}
	\caption{All regular LCDs of size $3$ where the two on top are 
    connected ones.\label{F:b3}}
\end{figure}
By exploiting the relation between the generating function
(\ref{E:genfish}) of the Fishburn numbers and the ``half derivative''
of the Dedekind eta-function, Zagier \cite{Zagier2001} derived the
asymptotic behavior of Fishburn numbers $f_n$, which is our
\eqref{E:an-counts} with $\lambda_1=\lambda_2=1$:
\begin{align}\label{E:zag1}
	f_n := [z^n]\sum_{k\ge0}\prod_{1\le j\le k}
	\lpa{1-(1-z)^{-j}}
	= cn^{\frac12} \mu^n n! 
	\lpa{1+O\lpa{n^{-1}}},
\end{align}
where $(c,\mu) := \lpa{\tfrac{12\sqrt{3}}{\pi^{5/2}}
\, e^{\frac{\pi^2}{12}}, \tfrac{6}{\pi^2}}$.

Given any regular LCD $D$, consider a graph $G_D$ (also known as
\emph{intersection graph}) whose vertices are arcs of $D$ and two
vertices are connected by an edge if the corresponding arcs cross
each other; in this case if $G_D$ is connected, then $D$ is also said
to be \emph{connected}; see Figure~\ref{F:b3} for an illustrative
example. The corresponding generating function $g(z)$ of connected
regular LCDs counted by size is given in \eqref{E:phi}.

We now prove Theorem \ref{T:genzag}, showing that the probability of
a uniformly generated large random regular LCD being connected is
asymptotic to $e^{-1}$.

\begin{proof}
We are going to prove
\begin{align}\label{E:equistom}
	\frac{[z^n]g(z)}{[z^n]\Phi(z,0)}
	= \frac{g_n}{f_n}
	=e^{-1}\lpa{1+O\lpa{n^{-1}}},
\end{align}
where $[z^n]\Phi(z,0)=f_n$ is the $n$-th Fishburn number. In view of
\eqref{E:zag1}, it remains to evaluate $g_n$ as $n$ tends to
infinity.

We first rewrite \eqref{E:phi} by using the two relations 
\eqref{E:fphi} and \eqref{E:Fzv4}, giving 
\begin{align}\label{E:phizv}
	\Phi(z,v)
	=\frac{v}{(1+v)^2}+\frac{1}{(1+v)^2}
	\sum_{k\ge0}\frac{1}{(1-z)^{k+1}}
	\prod_{1\le j\le k}\frac{((1-z)^{-j}-1)^2}
	{1+v(1-z)^{-j}}.
\end{align}
As a result, the equation $\Phi(z,g(z))=1$ can be written as 
\begin{align}\label{E:lambdaA}
	g(z)^2+g(z)+1
	=W(z) := \sum_{k\ge0}\frac{1}{(1-z)^{k+1}}
	\prod_{1\le j\le k}\frac{((1-z)^{-j}-1)^2}
	{1+g(z)(1-z)^{-j}}.
\end{align}
Taking the coefficients of $z^n$ on both sides yields
\begin{align}\label{E:eq1}
	[z^n]g(z)+[z^n]g(z)^2
	=[z^n]W(z),\qquad(n\ge1).
\end{align}
While this equation is still recursive, the dependence of the
first-order asymptotic approximation on $g$ is however weak, and
indeed only on the coefficient $g'(0)=1$ (see \eqref{E:gz-taylor}),
similar to the two examples in \cite[\S 6.1.4]{Hwang2020}. 

Technically, comparing the right-hand side of \eqref{E:lambdaA} with 
the expression (by \eqref{E:genfish} and \eqref{E:aj-q})
\[
    f_n = [z^n]\sum_{k\ge0}\frac{1}{(1-z)^{k+1}}
	\prod_{1\le j\le k}((1-z)^{-j}-1)^2,
\]
we see that the limiting constant $e^{-1}$ will come from the 
extra product
\[
    \bar{W}_k(z) 
    := \prod_{1\le j\le k} \frac1{1+g(z)(1-z)^{-j}}.
\]
For the analysis, we first truncate the series to a polynomial:
\[
    [z^n]\bar{W}_k(z)
    = [z^n]\prod_{1\le j\le k} \frac1{1+\bar{g}_n(z)(1-z)^{-j}},
\] 
where $\bar g_n(z) := \sum_{1\le \ell\le n}g_\ell z^\ell$. Then we 
prove that $\bar g_n(z)=O(|z|)$ when $|z|\le \frac{e}{\mu}-\ve$. By 
the trivial bound $g_n\le f_n$ and the estimate \eqref{E:zag1}, we 
have, when $|z|= \frac{\varrho}{n}$,
\[
    |\bar g_n(z)|\le \bar g_n(|z|) 
    = O\llpa{\sum_{1\le \ell \le n}
    f_\ell |z|^\ell}
    = O\llpa{\sum_{1\le \ell \le n}
    \sqrt{\ell}\,\mu^\ell \ell! \varrho^\ell n^{-\ell}},
\]
which is bounded above by $O(n^{-1})$:
\[
    O\llpa{\sum_{1\le \ell \le n}
    \ell\lpa{\varrho \mu e^{-1}\ell n^{-1}}^\ell}
    = O\lpa{n^{-1}},
\]
whenever $\varrho<\frac{e}{\mu}$. It follows that, with the same 
$|z|=\frac{\varrho}n$,
\[
    \bar W_k(z)
    = O\llpa{\Exp\Lpa{|z|\sum_{1\le j\le k}(1-|z|)^{-j}
    }} = O(1),
\]
whenever $\varrho<\frac{e}{\mu}$. Thus the extra product $\bar 
W_k(z)$ will only affect the constant term in the asymptotic analysis 
of $g_n$.

Indeed, from the proof of Theorem~\ref{T:dimdis} or
Proposition~\ref{P:Xn-cf}, it suffices to examine the behavior of
this finite product when $k=qn+O\lpa{n^{\frac12+\ve}}$ and
$z=(n\mu)^{-1}\lpa{1+O(n^{-\frac12})}$. Note that $q=\mu\Log
2<\frac{e}{\mu}$. Since $\bar g_n(z)=z+O(|z|^2)$ (when $z\asymp 
n^{-1}$), we then get, for such $k$ and $z$,
\begin{align*}
    \prod_{1\le j\le k} \frac1{1+\bar g_n(z)(1-z)^{-j}}
    &= \Exp\llpa{-\bar g_n(z)\sum_{1\le j\le k}(1-z)^{-j} 
    +O\llpa{|z|^2\sum_{1\le j\le k}|1-z|^{-2j}}}\\
    &= \Exp\llpa{-\lpa{z+O(|z|^2)}\frac{(1-z)^{-k}-1}{z}
    +O\lpa{|z|}}\\
    &= e^{-(e^{q/\mu}-1)}\lpa{1+O\lpa{n^{-\frac12}}}
    = e^{-1}\lpa{1+O\lpa{n^{-\frac12}}}.
\end{align*}
The detailed proof follows the same procedure we used for the proof 
of Theorem~\ref{T:dimdis}, and is omitted here. We thus obtain 
\begin{align}\label{E:bk}
    w_n 
    := [z^n]W(z)
    = c' n^{\frac12} \mu^n n! 
    \lpa{1+O\lpa{n^{-1}}},\WITH
    c' := \frac{12\sqrt{3}}{\pi^{5/2}}
    \, e^{\frac{\pi^2}{12}-1}.
\end{align}
This gives an approximation of the right-hand-side of (\ref{E:eq1}). 
It remains to prove that $g_n=[z^n]g(z)$ is asymptotic to $w_n$. 
Note that (\ref{E:eq1}) implies that 
\begin{align}\label{E:gk}
    [z^n]g(z) 
    &= \sum_{1\le \ell\le n}
    \frac{(-1)^{\ell-1}}{\ell}\binom{2\ell}{\ell}
    [z^n](W(z)-1)^\ell\\
    &= w_n - \sum_{1\le j<n}w_jw_{n-j}
    +2\sum_{\substack{j_1+j_2+j_3=n\\
    1\le j_1,j_2,j_3<n}}
    w_{j_1}w_{j_2}w_{j_3}-\cdots.\nonumber
\end{align}
Here the central binomial coefficients in (\ref{E:gk}) only increase 
exponentially, while $w_n$ grows factorially. Now, by \eqref{E:bk},
\begin{align*}
    \sum_{1\le j<n}w_jw_{n-j}
    &=2w_{n-1} +
	O\llpa{\sqrt{n}\,\mu^n
    \sum_{2\le j\le n-2}j!(n-j)!}\\
    &= O\lpa{n^{-1}w_n + \sqrt{n}\,\mu^n (n-1)!}\\
    &= O\lpa{n^{-1}w_n},
\end{align*}
because $j!(n-j)!$ decreases in $j\in[0,\frac12n]$. Similarly, 
\[
    [z^n](W(z)-1)^\ell = O\lpa{n^{-\ell}w_n},
    \qquad(\ell=1,2,\dots),
\]
showing that \eqref{E:gk} is itself an asymptotic expansion. Thus
\begin{align*}
	g_n
    = w_n\lpa{1+O\lpa{n^{-1}}},
\end{align*}
which, together with \eqref{E:bk} and \eqref{E:zag1}, proves the 
limiting ratio  \eqref{E:equistom}, and thus Theorem~\ref{T:genzag}.
\end{proof}


Finer approximations for the ratio $\frac{g_n}{f_n}$ can be derived; 
for example, we have 
\[
	\frac{g_n}{f_n}
	= e^{-1}\llpa{1-\frac{\pi^2}{8n}+O\lpa{n^{-2}}}.
\]

\section{Size distribution of random FMs}
\label{S:je}

We prove in this section Theorem \ref{T:dimdis2}, which is an
extension of the open problem 5.5 by Jel\'inek \cite{Jelinek2012}.
Unlike the analytic proof for Theorem~\ref{T:dimdis}, our approach to 
Theorem~\ref{T:dimdis2} builds on a simple partial fraction 
decomposition. We also briefly discuss a few other sequences of a 
similar nature.

\subsection{Asymptotic normality of the size}

We recall that the generating function $F(z,v)$ of $\Lambda$-FMs is
given by (\ref{E:Fzv4}). Our study of size distribution is restricted
to the situation when $\Lambda(z)$ is a polynomial. In the special
case when $\Lambda(z)=1+z$, the size of an $m$-dimensional primitive
FM lies between $m$ (when only the entries on the main diagonal are
$1$) and $\binom{m+1}2$ (when all entries are $1$). Only near the
median size $\frac14m(m+3)$ does the number of primitive FMs of
dimension $m$ reach its peak among all other possible sizes, which is
also the case when $0$'s and $1$'s are allowed to appear equally
likely in each entry (except the diagonal). For instance, in
Figure~\ref{F:fishburndim}, most primitive FMs of dimension $3$ have
size $4$ or $5$.

\begin{proof}
Let $Y_m$ be the size of a random $m\times m$ $\Lambda$-FM when all 
$\Lambda$-FMs of dimension $m$ are equally likely to be selected. The 
corresponding probability generating function of $Y_m$ is given by 
\begin{align*}
    \mathbb{E}(z^{Y_m})
    =\frac{[v^m]F(z,v)}{[v^m]F(1,v)},
\end{align*}
with $F$ given in \eqref{E:Fzv4}. By partial fraction expansion, 
\[
    \prod_{1\le j\le k}
    \frac{\Lambda(z)^j-1}{1+v(\Lambda(z)^j-1)}
    = \sum_{1\le j\le k}
    \frac{C_{k,j}(z)}{1+v(\Lambda(z)^j-1)},
\]
where
\begin{align*}
    C_{k,j}(z) 
    &:= \lpa{\Lambda(z)^j-1}^k
    \prod_{l\ne j, 1\le l\le k}\frac{\Lambda(z)^l-1}
    {\Lambda(z)^j-\Lambda(z)^l}\\
    &= \lpa{\Lambda(z)^j-1}^k
    \llpa{\prod_{1\le l<j}\frac{\Lambda(z)^l-1}
    {\Lambda(z)^j-\Lambda(z)^l}}
    \llpa{\prod_{j<l\le k}\frac{\Lambda(z)^l-1}
    {\Lambda(z)^j-\Lambda(z)^l}}\\
    &= (-1)^{k-j}\lpa{\Lambda(z)^j-1}^{k} 
    \Lambda(z)^{-\binom{j}2}
    \prod_{1\le l\le k-j}\frac{1-\Lambda(z)^{-j-l}}
	{1-\Lambda(z)^{-l}}.
\end{align*}
Consequently,
\begin{align*}
    [v^m]F(z,v) 
    &= \sum_{1\le k\le m}\sum_{1\le j\le k}
    C_{k,j}(z) (-1)^{m-k}
    \lpa{\Lambda(z)^j-1}^{m-k}\\
    &= \sum_{1\le k\le m}\sum_{1\le j\le k}
    (-1)^{m-j}\lpa{\Lambda(z)^j-1}^m\Lambda(z)^{-\binom{j}2}
    \llpa{\prod_{1\le l\le k-j}\frac{1-\Lambda(z)^{-l-j}}
	{1-\Lambda(z)^{-l}}}.
\end{align*}
Interchanging the two sums and rearranging the sum-indices lead to 
\begin{align*}
	[v^m]F(z,v) 
    &= \Lambda(z)^{\binom{m+1}2}
    \sum_{0\le j< m}
    (-1)^{j}\lpa{1-\Lambda(z)^{-m+j}}^m
	\Lambda(z)^{-\binom{j+1}2}
    \sum_{0\le k\le j}
	\prod_{1\le l\le k}\frac{1-\Lambda(z)^{-l-m+j}}
	{1-\Lambda(z)^{-l}}.
\end{align*} 
For our limit law purposes, we consider $z\sim1$. Since $\Lambda(z)$
is a polynomial with positive coefficients and $\Lambda(1)>1$, there
is a small neighborhood of unity, say $|z-1|\le\delta$, $\delta>0$,
where $|\Lambda(z)|>1$. Thus for such $z$
\[
    \lpa{1-\Lambda(z)^{-m+j}}^m-1
    = O\lpa{m|\Lambda(z)|^{-m+j}}.
\]
We then deduce that 
\begin{align}\label{E:Ym-qpa}
    [v^m]F(z,v)&= H(z) \Lambda(z)^{\binom{m+1}2} 
    \lpa{1+O\lpa{m|\Lambda(z)|^{-m}}},
\end{align}
uniformly for $|z-1|\le\delta$, where
\begin{align*}
    H(z):= \sum_{j\ge0}(-1)^{j}
	\Lambda(z)^{-\binom{j+1}2}
    \sum_{0\le k\le j}
	\prod_{1\le l\le k}\frac1
	{1-\Lambda(z)^{-l}}.
\end{align*}
In particular, the total number of $m$-dimensional $\Lambda$-FMs 
satisfies 
\begin{align}\label{E:m-dim-fm}
    \frac{[v^m]F(1,v)}{\Lambda(1)^{\binom{m+1}2}} 
    \to H(1)  = \sum_{j\ge0}(-1)^{j}\Lambda(1)^{-\binom{j+1}2}
    \sum_{0\le k\le j}\frac1{Q_k},
\end{align}
where
\[
	Q_k := \prod_{1\le \ell \le k}\frac1{1-\Lambda(1)^{-\ell}}.
\]
An alternative expression of $H(1)$ with additional numerical 
advantages is 
\[
	H(1) 
	= \frac1{Q_\infty}
	\sum_{j\ge0}(-1)^{j}\Lambda(1)^{-\binom{j+1}2}
    \sum_{l\ge0}\frac{(-1)^j\Lambda(1)^{-\binom{l}2}}{Q_l}
	\cdot\frac{1-\Lambda(1)^{-l(j+1)}}{1-\Lambda(1)^{-l}},
\]
which follows from the Euler identity
\[
	\frac{Q_\infty}{Q_k}
	= \sum_{l\ge0}\frac{(-1)^l\Lambda(1)^{-\binom{l}2}}
	{Q_l}\,\Lambda(1)^{-kl}. 
\]

Furthermore, it also follows from \eqref{E:Ym-qpa} that
\begin{align}\label{E:EzYm}
    \mathbb{E}(z^{Y_m})
    =\frac{[v^m]F(z,v)}{[v^m]F(1,v)}
    =\frac{H(z)}{H(1)}
    \left(\frac{\Lambda(z)}{\Lambda(1)}\right)^{\binom{m+1}2}
    \lpa{1+O\lpa{m\Lambda(1)^{-m}+m|\Lambda(z)|^{-m}}},
\end{align}
uniformly for $|z-1|\le\delta$. By applying the Quasi-powers Theorem 
(\cite[IX.5]{Flajolet2009} or \cite{Hwang1998}), we conclude that the 
distribution of the random variable $Y_m$ is asymptotically normally 
distributed  with mean and variance asymptotic to 
\begin{align*}
    \mathbb{E}(Y_m)
    &=\hat{\mu}m(m+1)+\frac{H'(1)}{H(1)}+O\lpa{m\Lambda(1)^{-m}},\\
    \mathbb{V}(Y_m)
    &= \hat{\sigma}^2m(m+1)+\frac{H'(1)+H''(1)}{H(1)}
    -\Lpa{\frac{H'(1)}{H}}^2+O\lpa{m\Lambda(1)^{-m}},
\end{align*}
where $(\hat{\mu}, \hat{\sigma}^2)$ are defined in 
\eqref{E:mup-sigmap}.
\end{proof}

As a special case, consider $\Lambda=\{0,1,\dots,h-1\}$, where 
$h\ge2$. Then $\Lambda(1) = h$. The asymptotic expression 
\eqref{E:m-dim-fm} then suggests the following algorithm for 
generating a random $m$-dimensional $\Lambda$-FM: Generate first the 
two corners on the diagonal by two independent integer-valued uniform 
distribution $\text{Uniform}[1,h-1]$, and then the remaining 
entries $\text{Uniform}[0,h-1]$. Reject the matrix if it fails to 
be Fishburn and stop if it is. The probability of success is given 
by, according to \eqref{E:m-dim-fm},
\[
	p_h := \frac{h^2}{(h-1)^2}
	\sum_{j\ge0}(-1)^{j}h^{-\binom{j+1}2}
	\sum_{0\le k\le j}\prod_{1\le l\le k}\frac1
	{1-h^{-l}}.
\]
This probability tends to $1$ as $h$ increases. 
\begin{center}
\begin{tabular}{c|ccccccccc}
	$h$ & $2$ & $3$ & $4$ & $5$ & $6$ 
	& $7$ & $8$ & $9$ & $10$ \\ \hline
	$p_h$ & $0.334$ & $0.706$ & $0.843$ & $0.903$ & $0.935$
	& $0.953$ & $0.965$ & $0.972$ & $0.978$
\end{tabular}	
\end{center}
The more than doubled jump of the success probability from $p_2$ to 
$p_3$ is more significant than expected; these values also show that 
the naive rejection method is generally very efficient. 

Three sequences are found in the OEIS of the form $[v^m]F(1,v)$, and
they are summarized in the following table; their asymptotic
behaviors are described by \eqref{E:m-dim-fm}.

\begin{center}
\begin{tabular}{c|ccc}
OEIS & 
\cite[\href{https://oeis.org/A005321}{A005321}]{oeis2019} & 
\cite[\href{https://oeis.org/A289314}{A289314}]{oeis2019} &
\cite[\href{https://oeis.org/A289315}{A289315}]{oeis2019} \\ \hline
$\Lambda(z)$ & $1+z$ &  $1+z+z^2$ & $1+z+z^2+z^3$ \\
Generating function $F(1,z)$ & 
$\sum\limits_{k\ge0}
\prod\limits_{1\le j\le k}\frac{(2^j-1)z}{1+(2^j-1)z}$ &
$\sum\limits_{k\ge0}
\prod\limits_{1\le j\le k}\frac{(3^j-1)z}{1+(3^j-1)z}$ &
$\sum\limits_{k\ge0}
\prod\limits_{1\le j\le k}\frac{(4^j-1)z}{1+(4^j-1)z}$
\end{tabular}	
\end{center}

\subsection{Andresen and Kjeldsen's 1976 paper}
\label{S:AK-1976}

In a somewhat disguised context of transitively directed graphs,
Andresen and Kjeldsen studied in their pioneering paper
\cite{Andresen1976} three sequences connected to primitive FMs,
denoted by $\xi_{m,k}, \eta_{m,k}$ and $\psi_{m,k}$, respectively.
(We change their notation $f(m,k)$ to $f_{m,k}$ to reduce the
occurrences of parentheses.) We show in this subsection that these
three sequences are all asymptotically normally distributed for large
$m$ with mean and variance asymptotic to $\frac12m$ and $\frac 14m$,
respectively.

In terms of the matrix language, $\xi_{m,k}$ counts the number of
primitive FMs of dimension $m$ with first row sum $k$, and
$\psi_{m,k}$ the number of upper triangular binary matrices (matrices
with entries $0$ or $1$) of dimension $m$ with first row sum $k$ such
that
\begin{itemize}
\item the $j$-th column $(1\le j<m)$ is a zero column if and only if
the $(j+1)$-st row is a zero row;

\item all nonzero columns and rows form a primitive FM.
\end{itemize}
While a combinatorial interpretation of the last sequence 
$\eta_{m,k}$ is still lacking (an open question in 
\cite{Andresen1976}), its generating polynomial satisfies a similar 
type of recurrence as that of $\xi$ and $\psi$:
\[
    \left\{
    \begin{split}
        P_m^{[\xi]}(v) &= vP_{m-1}^{[\xi]}(1+2v) 
        -v P_{m-1}^{[\xi]}(v),\\
        P_m^{[\eta]}(v) &= vP_{m-1}^{[\eta]}(1+2v) 
        -(1+v) P_{m-1}^{[\eta]}(v),\\
        P_m^{[\psi]}(v) &= vP_{m-1}^{[\psi]}(1+2v) 
        +(1-v) P_{m-1}^{[\psi]}(v),
    \end{split}
    \right.
\]
for $n\ge2$, all with the same initial conditions
$P_1^{[\cdot]}(v)=v$. Other types of recurrences are also derived in
\cite{Andresen1976}. These recurrences are readily solved by
iterating the corresponding functional equations satisfied by the
bivariate generating functions, and we obtain
\begin{center}
\renewcommand{\arraystretch}{1.5}    
\begin{tabular}{c|ccc}
    Sequence & $\xi_{m,k}$ & $\eta_{m,k}$ & $\psi_{m,k}$\\ \hline 
	OEIS 
	& \cite[\href{https://oeis.org/A259971}{A259971}]{oeis2019}
	& \cite[\href{https://oeis.org/A259972}{A259972}]{oeis2019}
	& \cite[\href{https://oeis.org/A259970}{A259970}]{oeis2019} \\
    Bivariate GF & 
	$\sum\limits_{k\ge0}\prod\limits_{0\le j\le k}
    \frac{(2^j-1+2^jv)z}{1-z+2^j(1+v)z}$
    & $\sum\limits_{k\ge0}\prod\limits_{0\le j\le k}
    \frac{(2^j-1+2^jv)z}{1+2^j(1+v)z}$
    & $\sum\limits_{k\ge0}\prod\limits_{0\le j\le k}
    \frac{(2^j-1+2^jv)z}{1-2z+2^j(1+v)z}$ \\ \hline
	$\sum_k f_{m,k}$ 
	& \cite[\href{https://oeis.org/A005321}{A005321}]{oeis2019}
    & \cite[\href{https://oeis.org/A005014}{A005014}]{oeis2019}
    & \cite[\href{https://oeis.org/A005016}{A005016}]{oeis2019}
	\\ 
\end{tabular}    
\end{center}
\medskip

In particular, by a direct partial fraction expansion, 
\begin{align*}
    P_m^{[\eta]}(v) &= v\sum_{0\le k<m}
    (-1)^k(1+v)^{m-1-k}\prod_{k<l<m}(2^l-1)
	\qquad(m\ge1),
\end{align*}
and the reason of introducing $\eta_{m,k}$ is because of the relations
\[
    P_m^{[\xi]}(v) 
    = \sum_{0\le j<m}\binom{m-1}{j}P_{m-j}^{[\eta]}(v),
    \quad\text{and}\quad
    P_n^{[\psi]}(v) 
    = \sum_{0\le j<m}\binom{m-1}{j}2^jP_{m-j}^{[\eta]}(v).
\]
From these forms, the limiting normal distribution in all cases can 
be derived by a similar argument used above for the size distribution 
$Y_m$. Now write $Q_m := \prod_{1\le j\le m}(1-2^{-j})$. Then 
\begin{align*}
    P_m^{[\eta]}(v) &= 2^{\binom{m}2}\sum_{0\le k<m}
    \frac{(-1)^k2^{-\binom{k}2}Q_{m-1}}{Q_k}\,(1+v)^{m-1-k}\\
    &= T(v)v(1+v)^{m-1}2^{\binom{m}2}
    \llpa{1+ O\Lpa{2^{\binom{m}2}\sum_{k\ge m}
    2^{-\binom{k}2}|1+v|^{-k}}},
\end{align*}
where
\[
    T(v) := Q_\infty \sum_{k\ge0}
    \frac{(-1)^k2^{-\binom{k}2}}{Q_k}\,(1+v)^{-k},
\]
is a meromorphic function of $v$. This implies that 
\[
    \frac{P_m^{[\eta]}(v)}{P_m^{[\eta]}(1)}
    = \frac{vT(v)}{T(1)}\Lpa{\frac{1+v}{2}}^{m-1}
    \lpa{1+O\lpa{|1+v|^{-m}}},
\]
uniformly for $|v+1|\ge 1+\ve$, and from this we then deduce the 
asymptotic normality $\mathscr{N} \lpa{\frac12m, \frac14m}$ for the 
underlying random variables by Quasi-powers Theorem 
(\cite[IX.5]{Flajolet2009} or \cite{Hwang1998}). Exactly the same 
type of results hold for the other two sequences.

\subsection{Some related OEIS sequences}

A few other sequences in the OEIS are closely connected to the
sequences we discussed in this section. We list them in the following
table. Asymptotic or distributional properties can be dealt with by
the same techniques, and are omitted here.

\begin{center}
\renewcommand{\arraystretch}{1.5}    	
\begin{tabular}{cc|cc}\hline

\cite[\href{https://oeis.org/A005327}{A005327}]{oeis2019}
& $\frac1{1+z}\sum\limits_{k\ge0}
   \prod\limits_{1\le j\le k}\frac{(2^j-1)z}{1+2^jz}$ 
& \cite[\href{https://oeis.org/A002820}{A002820}]{oeis2019}
& $2^{\binom{n}{2}}\times \text{A005327}(n+1)$ \\  

\cite[\href{https://oeis.org/A005016}{A005016}]{oeis2019}
& $\sum\limits_{k\ge0}\prod\limits_{1\le j\le k}
\frac{(2^j-1)z}{1+(2^j-2)z}$  
& \cite[\href{https://oeis.org/A005331}{A005331}]{oeis2019}
& $\frac1{1-z}
\sum\limits_{k\ge0}\prod\limits_{1\le j\le k}
\frac{(2^j-1)z}{1+(2^j-2)z}$ \\ 

\cite[\href{https://oeis.org/A005329}{A005329}]{oeis2019}
& $\sum\limits_{k\ge0}\prod\limits_{1\le j\le k}\frac{2^jz}{1+2^jz}$
& \cite[\href{https://oeis.org/A028362}{A028362}]{oeis2019}
& $\sum\limits_{k\ge0}\prod\limits_{0\le j< k}\frac{2^jz}{1-2^jz}$\\  

\cite[\href{https://oeis.org/A182507}{A182507}]{oeis2019}
& $\sum\limits_{k\ge0}\prod\limits_{1\le j\le k}
\frac{j2^{j-1}z}{1+j2^jz}$  
& \cite[\href{https://oeis.org/A006116}{A006116}]{oeis2019}:
& $\sum\limits_{k\ge0}z^k\prod\limits_{0\le j\le k}
\frac{1}{1-2^jz}$ \\ \hline
\end{tabular}	
\end{center}
\begin{remark}
It is interesting to compare 
\cite[\href{https://oeis.org/A028362}{A028362}]{oeis2019} with the 
two identities
\begin{align*}
	1+z = \sum\limits_{k\ge0}\prod\limits_{1\le j\le k}
	\frac{2^{j-1}z}{1+2^jz}, \quad\text{and}\quad
	\frac1{1-z} = \sum\limits_{k\ge0}\prod\limits_{1\le j\le k}
	\frac{jz}{1+jz}.
\end{align*}
Another sequence with a similar nature is the Gaussian 
polynomials 
\cite[\href{https://oeis.org/A022166}{A022166}]{oeis2019}: 
$$\sum\limits_{k\ge0}\prod\limits_{1\le j\le k}\frac{vz}{1-2^jz}.$$

\end{remark}
\section{Self-dual FMs and FMs without $1$'s}
\label{S:final}

In this section, we briefly describe the limiting behaviors of random
self-dual FMs and random FMs whose smallest nonzero entries are $2$,
respectively. The asymptotics in both cases are similar and involve a
stretched exponential factor of the form $e^{\Theta(\sqrt{n})}$. We
only sketch the proof in the self-dual case, and omit that in the
other.

\subsection{Dimension of self-dual FMs}

The dimension distribution of $\Lambda$-FMs (Theorem~\ref{T:dimdis})
exhibits a limiting invariance property in the sense that the central
limit theorem is independent of the entry-set $\Lambda$ as long as
$\lambda_1>0$. We show here that the same limiting property holds
even when we restrict our random matrices to be self-dual (or
persymmetric). What is less expected here is that the variance in the
random self-dual FMs is asymptotically double that in the ordinary
case (while the mean remains asymptotically the same); see
Table~\ref{T:self-dual} for a numerical illustration in the case of
primitive FMs (with $\Lambda(z) = 1+z$).

\begin{table}[!ht]
\begin{center}
\begin{tabular}{cc}
Self-dual primitive FMs & Primitive FMs \\
\begin{tabular}{c|cccccccc}
$n\backslash k$ & $1$ & $2$ & $3$ & $4$ & $5$ & $6$ & 
$(\mu_n,\sigma_n^2)$ \\ \hline 
$1$ & $1$ &&&&&& $(1,0)$ \\
$2$ & & $1$ &&&&& $(2,0)$\\
$3$ & & $1$ & $1$ &&&& $(\frac52,\frac14)$ \\
\rowcolor{blue!20}$4$ & & & $2$ & $1$ &&& 
$(\frac{10}3,\frac{2}{9})$ \\
$5$ & & & $2$ & $3$ & $1$ && $(\frac{23}{6},\frac{17}{36})$\\
$6$ & & & $1$ & $5$ & $6$ & $1$ & $(\frac{59}{13},\frac{94}{169})$\\
\end{tabular}    
&
\begin{tabular}{c|cccccccc}
$n\backslash k$ & $1$ & $2$ & $3$ & $4$ & $5$ & $6$ & 
$(\mu_n,\sigma_n^2)$\\ \hline 
$1$ & $1$ &&&&&& $(1,0)$\\
$2$ & & $1$ &&&&& $(2,0)$\\
$3$ & & $1$ & $1$ &&&& $(\frac52,\frac14)$ \\
\rowcolor{gray!40}$4$ & & & $4$ & $1$ &&& $(\frac{16}5,\frac{4}{25})$\\
$5$ & & & $4$ & $11$ & $1$ && $(\frac{61}{16},\frac{71}{256})$\\
$6$ & & & $1$ & $33$ & $26$ & $1$ 
& $(\frac{271}{61},\frac{1162}{3721})$\\
\end{tabular} 
\end{tabular}    
\end{center}
\medskip
\caption{The first few values of the dimension statistics in
self-dual primitive FMs and ordinary primitive FMs (where
$n\backslash k=$ size$\backslash$dimension). In particular, among the
$5$ primitive FMs of size $4$, only $3$ are self-dual, resulting in
higher variance, and a similar observation holds for matrices of
larger size.}\label{T:self-dual}
\end{table}

\begin{theorem} \label{T:dimdis4}
Let $Z_n$ denote the dimension of a random self-dual $\Lambda$-FM,
where all size-$n$ self-dual $\Lambda$-FMs are equally likely. If 
$\Lambda(z)$ is analytic at $z=0$ with $\lambda_1>0$, then $Z_n$ is 
asymptotically normally:
\[
    \frac{Z_n-\mu n-\mu'\sqrt{n}}{\sigma\sqrt{2n}}
    \stackrel{d}{\to} \mathscr{N}(0,1),
\]
where $(\mu,\sigma)$ is defined in \eqref{E:dimdis}, 
\[
    \mu' := \frac{\sqrt{6}}{2\pi^3}
    \lpa{12\Log 2-\pi^2\sqrt{\lambda_1}},
\]
and the mean and the variance are asymptotic to 
\begin{equation}\label{E:sd-mu-var}
    \begin{split}
	    \mathbb{E}(Z_n)
	    &= \mu n +\mu' \sqrt{n} +O(1),\\
	    \mathbb{V}(Z_n)
	    &= 2\sigma^2n+\frac{\sqrt{6}}{4\pi^5}
        \lpa{24(18-\pi^2)\Log 2-\pi^2(24-\pi^2)\sqrt{\lambda_1}}
        \sqrt{n}+O(1),
    \end{split}
\end{equation}
respectively.
\end{theorem} 
\begin{proof}
Our analysis is based on the generating function $G(z,v)$ for the
dimension (marked by $v$) of self-dual primitive FMs of a given size
(as marked by $z$) derived by Jel\'{i}nek in \cite{Jelinek2012}:
\begin{align} \label{E:Gzv-v}
	G(z,v)+v
	&=\sum_{k\ge1}
    \Lambda(z)^k\Lambda(z^2)^{\binom{k}{2}}v^{2k-1}
	\frac{1+v(\Lambda(z^2)^{k}-1)}{\Lambda(z^2)^{k}-1}
	\prod_{1\le j\le k}\frac{\Lambda(z^2)^{j}-1}
	{1+v^2(\Lambda(z^2)^{j}-1)}.
\end{align}
When $v=1$, we have  
\[
    G(z,1)+1 = \sum_{k\ge1}G_k(z),\WITH
    G_k(z) := \Lambda(z)^k
    \prod_{1\le j<k}\lpa{\Lambda(z^2)^j-1}.
\]
Asymptotic approximation of $[z^n]G(z,1)$ was already derived in 
\cite{Hwang2020}; in particular, when $\lambda_1>0$,
\[
	[z^n]G(z,1) = ce^{\beta\sqrt{n}}(\lambda_1\mu)^{\frac12n}
	n^{\frac12(n+1)}\lpa{1+O\lpa{n^{-\frac12}}},
\]
where 
\[
	(c,\beta,\mu) := 
	\llpa{\tfrac{3\sqrt{2}}{\pi^{3/2}}\,
    2^{\frac{\lambda_2}{\lambda_1}-\frac{\lambda_1}{2}}
    e^{-\frac{\lambda_1}{4}-\frac{\pi^2}{24}
    +\frac{\pi^2\lambda_2}{12\lambda_1^2}
    +\frac{3\lambda_1}{2\pi^2}(\Log 2)^2},
	\frac{\sqrt{6\lambda_1}}{\pi}\,\Log 2,
	\frac{6}{e\pi^2}}.
\]
A finer expansion can be derived, which is of the form 
\begin{align}\label{E:sd-ae}
	[z^n]\sum_{k\ge0}\Lambda(z)^k\prod_{1\le j<k}
	\lpa{\Lambda(z^2)^j-1}
	= ce^{\beta\sqrt{n}}\mu^{\frac12n}
	n^{\frac12n+1}\llpa{1+\sum_{j\ge1}
	\bar{d}_jn^{-\frac12j}},
\end{align}
for some computable coefficients $\bar d_j$. We compute first the
mean. For convenience, write $\Lambda_j := \Lambda(z^j)$. Taking the  
derivative with respect to $v$ and substituting $v=1$ on both sides 
of \eqref{E:Gzv-v} give
\begin{align*}
    M_1(z) &:= \partial_v G(z,v)|_{v=1}
    = \sum_{k\ge1}G_k(z)
    \Lpa{\frac{2}{\Lambda_2-1} - 
    \frac{\Lambda_2^{-k}\lpa{\Lambda_2+1}}
    {\Lambda_2-1}}-1\\
    &= \frac{2(G(z,1)+1)}{\Lambda_2-1}
    -\frac{\Lambda_2+1}{\Lambda_2-1}
	\sum_{k\ge1} \frac{\Lambda_1^k}{\Lambda_2^k}
    \prod_{1\le j< k}\lpa{\Lambda_2^j-1}-1.
\end{align*}
Let now
\begin{align*}
    S_1(z) &:= \sum_{k\ge1}\frac{\Lambda_1^k}{\Lambda_2^k}
    \prod_{1\le j< k}\lpa{\Lambda_2^j-1}
    = \sum_{k\ge0}\frac{\Lambda_1^{k+1}}{\Lambda_2^{k+1}}
    \prod_{1\le j\le k}\lpa{\Lambda_2^j-1}\\
	&= \frac{\Lambda_1}{\Lambda_2}+
	\frac{\Lambda_1}{\Lambda_2}
	\sum_{k\ge1}\frac{\Lambda_1^k(\Lambda_2^k-1)}{\Lambda_2^k}
	\prod_{1\le j< k}\lpa{\Lambda_2^j-1}\\
	&= \frac{\Lambda_1}{\Lambda_2}(G(z,1)+2)
	- \frac{\Lambda_1}{\Lambda_2}\,S_1(z).
\end{align*}
Thus
\[
    S_1(z) = \frac{\Lambda_1}{\Lambda_1+\Lambda_2}(G(z,1)+2),
\]
and, consequently,
\[
    M_1(z)
    = \frac{2\Lambda_2-\Lambda_1(\Lambda_2-1)}
	{(\Lambda_2-1)(\Lambda_1+\Lambda_2)}(G(z,1)+2)
	-\frac{2}{\Lambda_2-1}-1.
\]
By \eqref{E:sd-ae}, we then deduce the asymptotic approximation of 
the mean, as that given in \eqref{E:sd-mu-var}. 

For the variance, we compute first the second moment. By the same 
argument, we have
\begin{align*}
    M_2(z) &:= \partial_v^2 G(z,v)|_{v=1} 
    + \partial_v G(z,v)|_{v=1}  \\ 
    &= \sum_{k\ge1}G_k(z)\llpa{
    \frac{4(\Lambda_2^2+1)\Lambda_2^{-2k}}
    {(\Lambda_2-1)(\Lambda_2^2-1)}
    +\frac{(\Lambda_2^2-2\Lambda_2-7)\Lambda_2^{-k}}
    {(\Lambda_2-1)^2}-\frac{4(\Lambda_2^2-2\Lambda_2-1)}
    {(\Lambda_2-1)(\Lambda_2^2-1)}
    }-1.
\end{align*}
Let
\[
	S_2(z) := \sum_{k\ge1}\frac{\Lambda_1^k}{\Lambda_2^{2k}}
    \prod_{1\le j< k}\lpa{\Lambda_2^j-1}.
\]
Then 
\begin{align*}
	S_2(z)
	&= \frac{\Lambda_1}{\Lambda_2^2}
	+\frac{\Lambda_1^2(\Lambda_2-1)}{\Lambda_2^4}
	+\frac{\Lambda_1^2}{\Lambda_2^3}
	\sum_{k\ge1}\Lambda_1^k
	\llpa{1-\frac{\Lambda_2+1}{\Lambda_2^{k+1}}+
	\frac1{\Lambda_2^{2k+1}}}
	\prod_{1\le j< k}\lpa{\Lambda_2^j-1}\\
	&= \frac{\Lambda_1}{\Lambda_2^2}
	+\frac{\Lambda_1^2(\Lambda_2-1)}{\Lambda_2^4}
	+\frac{\Lambda_1^2}{\Lambda_2^3}
	\llpa{G(z,1)+1-\frac{\Lambda_2+1}{\Lambda_2}\,S_1(z)
	+\frac{S_2(z)}{\Lambda_2}},
\end{align*}
which is the solved to be
\[
	S_2(z) = \frac{\Lambda_1^2}
	{(\Lambda_1+\Lambda_2)(\Lambda_1+\Lambda_2^2)}\,
    (G(z,1)+2)
	+\frac{\Lambda_1}{\Lambda_1+\Lambda_2^2}.
\]
It follows that 
\begin{align*}
	M_2(z) &= \llpa{\frac{\Lambda_1}{\Lambda_1+\Lambda_2}
	+\frac{4\Lambda_2}{(\Lambda_2-1)^2(\Lambda_2^2-1)}
	\Lpa{\frac{3\Lambda_2^2-1}{\Lambda_1+\Lambda_2}
	-\frac{\Lambda_2^2(\Lambda_2^2+1)}
	{\Lambda_1+\Lambda_2^2}}}(G(z,1)+2)\\
	&\qquad +\frac{8\Lambda_2}{\Lambda_2^2-1}
	-\frac{4\Lambda_2^2(\Lambda_2^2+1)}
	{(\Lambda_2-1)(\Lambda_2^2-1)(\Lambda_1+\Lambda_2^2)}-1.
\end{align*}
From this expression and the expansion \eqref{E:sd-ae}, we 
deduce an asymptotic approximation to the second moment 
$\mathbb{E}(Z_n^2)$, and then the asymptotic variance in 
\eqref{E:sd-mu-var}.

The proof for the normal limit law is similar to that of 
Theorem~\ref{T:dimdis}, with the modifications needed to 
incorporate the change at the order $\sqrt{n}$. We list here the 
major steps. Prove first that when $\Lambda(z)=e^z$ and 
$v=e^{\theta/\sqrt{n}}$, 
\[
    [z^n]G(z,v) 
    = c  e^{\beta \sqrt{n}}
    \rho^{\frac12n}n^{\frac12(n+1)}
    e^{(\mu\sqrt{n}+\mu')\theta +\sigma^2\theta^2}
    \lpa{1+O\lpa{n^{-\frac12}}},
\]
where 
\[
    (c,\beta,\rho) = \llpa{\frac{3\sqrt{2}}{\pi^{3/2}}
    \,e^{-\frac14+\frac{3}{2\pi^2}(\Log 2)^2},
    \sqrt{\mu}\,\Log 2,\frac{\mu}{e}}.
\]
Then, by the change of variables $\Lambda(z^2)=e^{y^2}$, and by 
following the same analysis, we deduce that 
\[
    [z^n]G(z,v) 
    = c  e^{\beta \sqrt{n}}
    \rho^{\frac12n}n^{\frac12(n+1)}
    e^{(\mu\sqrt{n}+\mu')\theta +\sigma^2\theta^2}
    \lpa{1+O\lpa{(|\theta|+|\theta|^3)
	n^{-\frac12}}},
\]
uniformly for $|\theta|=o\lpa{n^{-1/6}}$, where 
\[
    (c,\beta,\rho) = \llpa{\tfrac{3\sqrt{2}}{\pi^{3/2}}\,
    2^{\frac{\lambda_2}{\lambda_1}-\frac{\lambda_1}{2}}
    e^{-\frac{\lambda_1}{4}-\frac{\pi^2}{24}
    +\frac{\pi^2\lambda_2}{12\lambda_1^2}
    +\frac{3\lambda_1}{2\pi^2}(\Log 2)^2},
    \sqrt{\mu\lambda_1}\Log 2,\frac{\mu\lambda_1}{e}}.
\]
This proves the asymptotic normality of $Z_n$.
\end{proof}

\subsection{FMs without $1$'s}

What happens if $\lambda_1=0$ and the smallest nonzero entry is 
$2$? In this case, the generating functions remain the same but with 
$\Lambda(z) = 1+\lambda_2z^2+\cdots$. Following the asymptotic 
approximations derived in \cite{Hwang2020}, we can also prove the 
corresponding central limit theorem for the dimension of random 
FMs.

\begin{theorem}\label{T:dimdis3}
Assume that $\Lambda(z)$ is analytic at $z=0$ with $\lambda_1=0$, 
$\lambda_2>0$ and that all such $\Lambda$-FMs of size $n$ are equally 
likely to be selected. Then the dimension $X_n$ of a random matrix is 
asymptotically normally distributed with mean and variance both 
linear in $n$:
\begin{align}\label{E:dimdis3}
	\frac{X_n-\bar\mu n-\bar\mu'\sqrt{n}}{\bar\sigma\sqrt{n}}
	\xrightarrow{d} \mathscr{N}(0,1),	
    \WITH
	(\bar\mu,\bar\mu',\bar\sigma^2)
    :=\llpa{\frac{3}{\pi^2},
    -\frac{\sqrt{3}\lambda_3}{2\pi\lambda_2^{3/2}},
    \frac{3(12-\pi^2)}{2\,\pi^4}},
\end{align}
so that $\bar\mu=\frac12\mu$ and $\bar\sigma^2=\frac12\sigma^2$, 
where $(\mu,\sigma^2)$ is given in \eqref{E:dimdis}, and 
\begin{align*}
	\mathbb{E}(X_n) &= \bar\mu n + \bar\mu'\sqrt{n}+O(1),\\
	\mathbb{V}(X_n) &= \bar{\sigma}^2n
	+\llpa{\frac{\lambda_3(\pi^2-6)}
	{4\pi^2\lambda_2^{3/2}}-\frac{2\lambda_5}{\lambda_2^{5/2}}
	+\frac{4\lambda_3\lambda_4}{\lambda_2^{7/2}}
	-\frac{2\lambda_3^3}{\lambda_2^{9/2}}}\sqrt{\bar\mu n}
	+O(1),
\end{align*}
respectively.
\end{theorem}
Both the mean and the variance constants are halved, when 
compared with those of the $\lambda_1>0$ case. This is 
intuitively clear as one expects that the entry $2$ is omnipresent. 

The analysis of this well anticipated limit result is much more
involved than it looks because the polynomial term in the asymptotic
approximation depends on the first nonzero odd number in the 
entry-set $\Lambda$. More precisely, if
\[
	\lambda_{2j-1}=0,\quad\text{for}\quad
	1\le j\le \ell, \AND \lambda_2,\lambda_{2\ell+1}>0,
\]
then it is proved in \cite{Hwang2020} that the total number of 
$\Lambda$-FMs of size $n$ satisfies 
\[
	[z^n]\sum_{k\ge0}\prod_{1\le j\le k}
	\lpa{1-\Lambda(z)^{-j}}
	= c_\ell e^{\beta \sqrt{n}}\rho^{\frac12n} 
	n^{\frac12n+\chi_\ell}
	\lpa{1+O\lpa{n^{-\frac12}}},
\]
where $(c_\ell,\chi_\ell)$ depends not only on $\ell$ but also on the 
parity of $n$, and 
\[
	(\beta,\rho) = 
	\llpa{\frac{\lambda_3\pi}{2\sqrt{3}\,\lambda_2^{3/2}},
	\frac{3\lambda_2}{e\pi^2}}.
\]

\section{Concluding remarks}

We conclude this paper by briefly indicating possible refinements to 
the central limit theorems derived in this paper. 

The simplest case is Theorem~\ref{T:dimdis2} concerning the size 
distribution of random FMs under fixed dimensions. Since it fits the 
standard Quasi-powers framework by \eqref{E:EzYm}, an optimal rate of 
order $O\lpa{m^{-1}}$ is readily guaranteed; see \cite{Hwang1998}. 

Optimal convergence rates in the other cases are structurally well
expected, but technically more involved. For instance, a closer
examination of our proof of Proposition~\ref{P:Xn-cf} shows that the
proof given there holds indeed in the wider range $|v-1|\le \ve$;
that is, we have the more precise estimate
\[
	\mathbb{E}\lpa{e^{X_ni\theta/\sqrt{n}}}
	= \begin{cases}
		e^{\mu\sqrt{n}i\theta-\frac12\sigma^2\theta^2}
		\lpa{1+O\lpa{(|\theta|+|\theta|^3)n^{-\frac12}}},
		&\text{if } |\theta|\le \ve n^{\frac16},\\
		O\lpa{e^{-\ve\theta^2}},
		&\text{if }\ve n^{\frac16}\le |\theta|\le\ve n^{\frac12}.
	\end{cases}
\]
Indeed, our proof of Proposition~\ref{P:Xn-cf} is simpler if we
restrict to the range
$|v-1|=O\lpa{n^{-\frac12}}$ for central limit theorem purposes. By
the classical Berry-Esseen inequality (see \cite[p.\
641]{Flajolet2009} or \cite{Hwang1998} and the references therein),
we can then obtain an optimal convergence rate of order
$O\lpa{n^{-\frac12}}$ in the central limit theorem \eqref{E:dimdis},
namely,
\[
    \sup_{x\in\mathbb{R}}\left|
    \mathbb{P}\llpa{\frac{X_n-\mu n}{\sigma \sqrt{n}}
    \le x}-\frac1{\sqrt{2\pi}}
    \int_{-\infty}^x e^{-\frac12t^2}\dd t \right| 
    = O\lpa{n^{-\frac12}}.
\] 
Similarly, optimal convergence rates can be derived for other central 
limit theorems, namely, Theorem~\ref{T:dimdis4} and 
Theorem~\ref{T:dimdis3}. 

Finally, local limit theorems are also anticipated, but the 
technicalities involved are more delicate; these and related 
approximations will be discussed elsewhere. 

\bibliographystyle{abbrv}
\bibliography{fishburn_dim}
\end{document}